\documentclass{article}
\usepackage{latexsym}
\usepackage{amstext}
\usepackage{amsmath}
\usepackage{amssymb}
\usepackage{amsthm}

\newtheorem{theorem}{Theorem}
\newtheorem{lemma}[theorem]{Lemma}
\newtheorem{corollary}[theorem]{Corollary}

\theoremstyle{definition}

\newtheorem{question}[theorem]{Question}
\theoremstyle{remark}
\newtheorem{remark}[theorem]{Remark}

\newcommand{\co}{{\mathcal O}}

\newcommand{\cN}{{\mathcal N}}

\newcommand{\Qbar}{\bar{\mathbb{Q}}}

\DeclareMathOperator{\ord}{ord}
\DeclareMathOperator{\Cl}{Cl}
\DeclareMathOperator{\rk}{rk}

\DeclareMathOperator{\Jac}{Jac}
\DeclareMathOperator{\Pic}{Pic}

\DeclareMathOperator{\Nk}{N_\mathbb{Q}^k}
\DeclareMathOperator{\Nki}{N_{\mathbb{Q}(P_i)}^{k(P_i)}}
\DeclareMathOperator{\NN}{N}

\begin{document}
\title{Ideal Class Groups and Torsion in Picard Groups of Varieties}
\author{Aaron Levin\\Centro di Ricerca Matematica Ennio De Giorgi\\
Collegio Puteano\\
Scuola Normale Superiore\\
Piazza dei Cavalieri, 3\\
I-56100 Pisa, Italy\\
E-mail: aaron.levin@sns.it
}
\date{}
\maketitle
\begin{abstract}
We give a new general technique for constructing and counting number fields with an ideal class group of nontrivial $m$-rank.  Our results can be viewed as providing a way of specializing the Picard group of a variety $V$ over $\mathbb{Q}$ to obtain class groups for number fields $\mathbb{Q}(P)$, $P\in V(\Qbar)$, for certain families of points $P$.  In particular, we show how the problem of constructing quadratic number fields with a large-rank ideal class group can be reduced to the problem of finding a hyperelliptic curve with a rational Weierstrass point and a large rational torsion subgroup in its Jacobian.  Furthermore, we show how many previous results on constructing large-rank ideal class groups can be fit into our framework and rederived.  As an application of our technique, we derive a quantitative version of a theorem of Nakano.  This gives the best known general quantitative result on number fields with a large-rank ideal class group.
\end{abstract}
\bibliographystyle{amsplain}
\section{Introduction}
The primary purpose of this paper is to give a new general technique for constructing and counting number fields with an ideal class group of nontrivial $m$-rank\footnote{For a finite abelian group $A$ and an integer $m>1$, we let $\rk_m A$, the $m$-rank of $A$, be the largest integer $r$ such that $(\mathbb{Z}/m\mathbb{Z})^r$ is a subgroup of $A$.}.  In spirit, our results can be viewed as providing a way of specializing the class group (Picard group) of a variety $V$ over $\mathbb{Q}$ to obtain class groups for number fields $\mathbb{Q}(P)$, $P\in V(\Qbar)$.  For instance, as an example of our results, we show how the problem of constructing quadratic number fields $k$ with a large-rank ideal class group $\Cl(k)$ can be reduced to the problem of finding a hyperelliptic curve $C$ with a rational Weierstrass point and a large rational torsion subgroup $\Jac(C)(\mathbb{Q})_{\rm{tors}}$ in its Jacobian.  More precisely, we prove:
\begin{theorem}
\label{chypintro}
Let $C$ be a nonsingular hyperelliptic curve over $\mathbb{Q}$ with a rational Weierstrass point.  Let $g$ denote the genus of $C$.  Let $m>1$ be an integer.  Then there exist $\gg X^{\frac{1}{2g+1}}/\log X$ imaginary quadratic number fields $k$ with discriminant $d_k$ and
\begin{equation*}
\rk_m \Cl(k)\geq \rk_m \Jac(C)(\mathbb{Q})_{\rm{tors}},\quad |d_k|<X,
\end{equation*}
and $\gg X^{\frac{1}{2g+1}}/\log X$ real quadratic number fields $k$ with
\begin{equation*}
\rk_m \Cl(k)\geq \rk_m \Jac(C)(\mathbb{Q})_{\rm{tors}}-1,\quad d_k<X.
\end{equation*}
\end{theorem}
More generally, we show that if $C$ is a nonsingular projective curve with $C(\mathbb{Q})\neq \emptyset$, then for most points $P$ in certain families of points on $C$, $\rk_m\Cl(\mathbb{Q}(P))$ can be bounded from below by a function of $\rk_m \Jac(C)(\mathbb{Q})_{\rm{tors}}$, the unit group $\co_{\mathbb{Q}(P)}^*$, and the set of places of bad reduction of the Jacobian of $C$.

In Section \ref{apps} we show how many of the previous results on constructing large-rank ideal class groups can be explained and rederived in the arithmetic-geometric context of this paper.  Thus, in the next section we give an overview and brief summary of the history of results on the problem of finding infinite families of number fields of degree $n$ over $\mathbb{Q}$ with ideal class groups of large $m$-rank.  We also obtain some new quantitative results on counting number fields with a large-rank ideal class group.  Of particular note is a quantitative version of a theorem of Nakano (Theorem \ref{QN}).

In \cite{Lev2}, by taking advantage of certain superelliptic curves, we obtained new results (see Theorems \ref{thLev1} and \ref{thLev2}) on class groups of large $m$-rank.  The proofs in \cite{Lev2} were given in a very explicit, concrete manner, avoiding any mention of the Jacobian or Picard group of a curve, but implicitly using the rational torsion in these objects.  The present paper is an abstraction and generalization of the technique used in \cite{Lev2}.

Fundamentally, in the simplest case, the basic idea is the following:  find polynomials $p,q\in k[x,y]$, for some number field $k$, such that if $i\in \mathbb{Z}$, $p(\alpha_i,i)=0$, and $\rk_m \Cl(\mathbb{Q}(\alpha_i))$ is small, then $q(x,i)$ is reducible over $k$.  By Hilbert's Irreducibility Theorem, we know that $q(x,i)$ is irreducible over $k$ for most values of $i\in \mathbb{Z}$, and therefore $\rk_m \Cl(\mathbb{Q}(\alpha_i))$ cannot be small for most values of $i$.  In a more geometric language, this construction can be accomplished by using torsion in the Picard group of a curve.  Let $C$ be a nonsingular projective curve over $\mathbb{Q}$ and let $\phi:C\to \mathbb{P}^1$ be a morphism.  Using $m$-torsion in $\Pic(C)$, the Picard group of $C$, we construct a certain curve $X$ and an unramified abelian cover $\pi:X\to C$.  This map has the property that if $\pi(Q)=P\in C(\Qbar)$, then the degree $[\mathbb{Q}(Q):\mathbb{Q}]$ is bounded by a function of the $m$-rank of $\Cl(\mathbb{Q}(P))$.  Consider points $Q_i\in X(\Qbar)$, $P_i\in C(\Qbar)$, with $\pi(Q_i)=P_i$ and $\phi(P_i)=i\in \mathbb{A}^1(\mathbb{Z})\subset \mathbb{P}^1(\mathbb{Q})$.  By Hilbert's Irreducibility Theorem, the degree $[\mathbb{Q}(Q_i):\mathbb{Q}]$ will usually be equal to $\deg \phi\circ\pi=(\deg \phi)(\deg \pi)$.  On the other hand, $[\mathbb{Q}(Q_i):\mathbb{Q}]$ is bounded by a function of $\rk_m \Cl(\mathbb{Q}(P_i))$.  Thus, by comparing these two statements we obtain number fields $\mathbb{Q}(P_i)$ with $[\mathbb{Q}(P_i):\mathbb{Q}]=\deg \phi$ and $\rk_m\Cl(\mathbb{Q}(P_i))$ bounded from below.  An alternative approach, replacing the use of Hilbert's Irreducibility Theorem by finiteness results on integral points of bounded degree on curves, was given in \cite{Lev2}.

Using this method, to obtain number fields of degree $n$ with ideal class groups of large $m$-rank, one needs to find $n$-gonal\footnote{Recall that a curve $C$ is called $n$-gonal if there exists a degree $n$ map from $C$ to $\mathbb{P}^1$.} curves $C$ over $\mathbb{Q}$ with $\rk_m \Pic(C)_{\rm{tors}}$ large.  This leads to the following natural question, which we state in terms of Jacobians\footnote{If $C(\mathbb{Q})\neq \emptyset$, then $\Pic(C)_{\rm tors}\cong \Jac(C)(\mathbb{Q})_{\rm tors}$;  in general, since a rational divisor class may not be represented by a rational divisor, there is only an inclusion $\Pic(C)_{\rm tors}\hookrightarrow\Jac(C)(\mathbb{Q})_{\rm tors}$.}.
\begin{question}
\label{Q1}
Let $n>1$ be a positive integer.  Let $p$ be a prime not dividing $n$.  Do there exist $n$-gonal curves $C$ over $\mathbb{Q}$ with $\rk_p \Jac(C)(\mathbb{Q})_{\rm{tors}}$ arbitrarily large?
\end{question}
\noindent In particular, we ask the following question of hyperelliptic curves:
\begin{question}
\label{Q2}
Let $p$ be an odd prime.  Do there exist hyperelliptic curves $C$ over $\mathbb{Q}$ with $\rk_p \Jac(C)(\mathbb{Q})_{\rm{tors}}$ arbitrarily large?
\end{question}
We have excluded the case when $p$ divides $n$ since in this case it is easy to construct $n$-gonal curves $C$ over $\mathbb{Q}$ with $\rk_p \Jac(C)(\mathbb{Q})_{\rm{tors}}$ arbitrarily large (analogously, when $p$ divides $n$ it is possible to construct number fields $k$ of degree $n$ with $\rk_p\Cl(k)$ arbitrarily large).

Previous papers studying the problem of constructing Jacobians of curves with large rational torsion subgroups have primarily focused on either curves of low genus \cite{Poo, Lep1, Lep2, Lep3, Lep4, Oga}, or on producing a rational torsion point of large order in the Jacobian of a curve of genus $g$, for every genus $g$ \cite{Fly1, Fly2, Lep6, Lep7, Lep8, Lep9}.  The main motivation in these papers seems to be to determine how Mazur's theorem classifying rational torsion in elliptic curves might generalize to curves of higher genus.  We hope to further motivate research into large rational torsion subgroups of Jacobians, though from the slightly different perspective of Questions \ref{Q1} and \ref{Q2}.

\section{Overview of Previous Results}
In this section we give a brief summary of the history of results on the problem of finding infinite families of number fields of degree $n$ over $\mathbb{Q}$ with ideal class groups of large $m$-rank (see Tables \ref{table1} and \ref{table2} for a more comprehensive list of results).  The earliest such result could be considered to be Gauss' result determining, in modern terms, the $2$-rank of the class group of a quadratic number field in terms of the primes dividing the discriminant of the quadratic field (see, e.g., \cite{Her}).  In particular, it follows from Gauss' result that the $2$-rank of the ideal class group of a quadratic number field can be made arbitrarily large.  In contrast to Gauss' result, there is not a single quadratic number field $k$ and prime $p\neq 2$ for which it is known that $\rk_p \Cl(k)>6$, although the Cohen-Lenstra heuristics \cite{Len, Len2} predict that for any given positive integer $r$, a positive proportion of quadratic fields $k$ should have $\rk_p \Cl(k)=r$.  
\begin{table}
\caption{Table of values $m$, $n$, and $r$ for which it is known that there exist infinitely many number fields $k$ of degree $n$ with $\rk_m\Cl(k)\geq r$ (with $r=\infty$ if $\rk_m\Cl(k)$ can be made arbitrarily large).\newline}
\centering
Quadratic fields ($n=2$)
\begin{tabular}{|l|c|c|c|c|}
\hline
Author(s) & Year & Type & $m$ & $r$\\
\hline
Gauss & 19th c. & imaginary, real & $2$ & $\infty$\\
\hline
Nagell \cite{Nag, Nag2}& 1922 & imaginary & $>1$ & $1$ \\
\hline
Yamamoto \cite{Yam} & 1970 & imaginary & $>1$& $2$\\
\hline
Yamamoto \cite{Yam}, Weinberger\cite{Wei} &1970, 1973 & real & $>1$& $1$\\
\hline
Craig \cite{Cra} & 1973 & imaginary & $3$& $3$\\
&  & real & $3$& $2$\\
\hline
Craig \cite{Cra2} & 1977 & imaginary & $3$& $4$\\
 & & real & $3$& $3$\\
\hline
Diaz \cite{Dia} & 1978 & real & $3$& $4$\\
\hline
Mestre \cite{Mes2, Mes, Mes3} & 1980 & imaginary, real & $5,7$& $2$\\
\hline
Mestre \cite{Mes4} & 1992 & imaginary, real & $5$& $3$\\
\hline
\end{tabular}
\newline\newline
Higher Degree Fields
\begin{tabular}{|l|c|c|c|c|}
\hline
Author(s) & Year & $m$ & $n$ & $r$\\
\hline
Brumer, Rosen \cite{Bru, Ros} & 1965 & $>1$ & $n=m$ & $\infty$\\
\hline
Uchida \cite{Uch} & 1974 & $>1$ & $3$ & $1$\\
\hline
Ishida \cite{Ish} & 1975 & $2$ & prime & $n-1$\\
\hline
Azuhata, Ichimura \cite{Ich} & 1984 & $>1$ & $>1$ & $\left\lfloor \frac{n}{2}\right\rfloor$\\
\hline
Nakano \cite{Nak5, Nak2} & 1985 & $>1$ & $>1$ & $\left\lfloor \frac{n}{2}\right\rfloor+1$\\
 &  & $2$ & $>1$ & $n$\\
\hline
Nakano \cite{Nak4} & 1988 & $2$ & $3$ & $6$\\
\hline
Levin \cite{Lev2} & 2006 & $>1$ & $>1$ & $\left\lceil\left\lfloor\frac{n}{2}\right\rfloor+\frac{n}{m-1}-m\right\rceil$\\
\hline
\end{tabular}
\label{table1}
\end{table}

The first constructive result on $m$-ranks of class groups for arbitrary $m$ was given in 1922 by Nagell \cite{Nag, Nag2}, who proved that for any positive integer $m$, there exist infinitely many imaginary quadratic number fields whose class group has an element of order $m$ (in particular, there are infinitely many imaginary quadratic fields with class number divisible by $m$).  Nagell's result has since been reproved by a number of different authors (e.g., \cite{Cho}, \cite{Hum}, \cite{Kur}).  Nearly fifty years later, working independently, Yamamoto \cite{Yam} and Weinberger \cite{Wei} extended Nagell's result to real quadratic fields.  Soon after, Uchida \cite{Uch} proved the analogous result for cubic cyclic fields.  In 1984, Azuhata and Ichimura \cite{Ich} succeeded in extending Nagell's result to number fields of arbitrary degree.  In fact, they proved that for any integers $m,n>1$ and any nonnegative integers $r_1$, $r_2$, with $r_1+2r_2=n$, there exist infinitely many number fields $k$ of degree $n=[k:\mathbb{Q}]$ with $r_1$ real places and $r_2$ complex places such that
\begin{equation}
\label{Ich}
\rk_m \Cl(k)\geq r_2.
\end{equation}
The right-hand side of (\ref{Ich}) was subsequently improved to $r_2+1$ by Nakano \cite{Nak5, Nak2}.  Choosing $r_2$ as large as possible, we thus obtain, for any $m$, infinitely many number fields $k$ of degree $n>1$ with
\begin{equation}
\label{eqN1}
\rk_m \Cl(k)\geq \left\lfloor\frac{n}{2}\right\rfloor+1,
\end{equation}
where $\lfloor \cdot \rfloor$ and $\lceil \cdot \rceil$ denote the greatest and least integer functions, respectively.  For general $m$ and $n$, \eqref{eqN1} is the best result that is known on producing number fields of degree $n$ with a class group of large $m$-rank.  In \cite{Lev2} it was shown that there exist infinitely many number fields $k$ of degree $n$ satisfying $\rk_m \Cl(k)\geq \left\lceil\left\lfloor\frac{n}{2}\right\rfloor+\frac{n}{m-1}-m\right\rceil$, improving \eqref{eqN1} when $n\geq m^2$.

For certain special values of $m$ and $n$, slightly more is known.  Of particular note to us are Mestre's papers \cite{Mes,Mes3,Mes2} giving the best known results for $m=5,7$ and $n=2$.  Although he uses a somewhat different approach, Mestre's method also crucially relies on rational torsion in Jacobians of curves.  Thus, his method lies in the same general vein as the present paper.
\begin{table}
\caption[test]{Table of $m$, $n$, $r$, and $f(X)$, for which it is known that $\cN_{m,n,r}(X)\gg f(X)$, where we let $d_k$ be the discriminant of $k$ over $\mathbb{Q}$ and
}
\centering
\begin{equation*}
\cN_{m,n,r}(X)=\#\{k\subset \bar{\mathbb{Q}}\mid [k:\mathbb{Q}]=n, \rk_m\Cl(k)\geq r, |d_k|<X\}.
\end{equation*}
Quadratic fields ($n=2$)
\begin{tabular}{|l|c|c|c|c|c|c|}
\hline
Author(s) & Year & Type & $m$ & $r$ & $f(X)$\\
\hline
Murty \cite{Mur2} & 1999 & imaginary & $>1$  & $1$ & $X^{\frac{1}{2}+\frac{1}{m}}$\\
& & real & odd  & $1$ & $X^{\frac{1}{2m}-\epsilon}$\\
\hline
Soundararajan \cite{Sou} & 2000 & imaginary & $m \equiv 0   \text{ (mod $4$)}$  & $1$ & $X^{\frac{1}{2}+\frac{2}{m}-\epsilon}$\\
 &  &  & $m \equiv 2   \text{ (mod $4$)}, m\neq 2$  & $1$ & $X^{\frac{1}{2}+\frac{3}{m+2}-\epsilon}$\\
\hline
Yu \cite{Yu} & 2002 & real & odd  & $1$ & $X^{\frac{1}{m}-\epsilon}$\\
\hline
Luca \cite{Luca2} & 2003 & real & even  & $1$ & $X^{\frac{1}{m}}/\log X$\\
\hline
Chakraborty, Murty \cite{Mur} & 2003 & real & 3  & $1$ & $X^{\frac{5}{6}}$\\
\hline
Byeon \cite{Bye2} & 2003 & real & 3  & $1$ & $X^{\frac{7}{8}}$\\
\hline
Byeon \cite{Bye} & 2006 & real & 5,7  & $1$ & $X^{\frac{1}{2}}$\\
\hline
Byeon \cite{Bye3} & 2006 & imaginary & odd  & $2$ & $X^{\frac{1}{m}-\epsilon}$\\
\hline
\end{tabular}
\newline\newline
Higher Degree Fields
\begin{tabular}{|l|c|c|c|c|c|}
\hline
Author(s) & Year & $m$ & $n$ & $r$ & $f(X)$\\
\hline
Hern{\'a}ndez, Luca \cite{Luca} & 2004 & $>1$ & $3$ & $1$ & $X^{\frac{1}{6m}}$\\
\hline
Bilu, Luca \cite{Bilu} & 2005 & $>1$ & $>1$ & $1$ & $X^{\frac{1}{2m(n-1)}}$\\
\hline
Levin \cite{Lev2} & 2006& $>1$ & $>1$ & $\left\lfloor \frac{n}{2}\right\rfloor$& $X^{\frac{1}{m(n-1)}}/\log X$\\
& &$>1$ & $>(m-1)^2$ & $\left\lceil\left\lfloor\frac{n}{2}\right\rfloor+\frac{n}{m-1}-m\right\rceil$& $ X^{\frac{1}{(m+1)n-1}}/\log X$\\
\hline
\end{tabular}
\label{table2}
\end{table}

Recently, progress has been made on obtaining quantitative results on counting the number fields in the above results.  Murty \cite{Mur2} gave the first results in this direction, obtaining quantitative versions of the theorems of Nagell and Yamamoto-Weinberger.  His results have since been improved by, among others, Soundararajan \cite{Sou} in the imaginary quadratic case and Yu \cite{Yu} in the real quadratic case.  In higher degrees, Hern{\'a}ndez and Luca \cite{Luca} gave the first such result for cubic number fields, while Bilu and Luca \cite{Bilu} succeeded in proving a quantitative theorem for number fields of arbitrary degree.  Bilu and Luca's result was improved in \cite{Lev2}, where a quantitative version of Azuhata and Ichimura's result was given.  We give a summary of these and other quantitative results in Table \ref{table2}.

\section{Hilbert's Irreducibility Theorem}
The main tool in our proofs is Hilbert's Irreducibility Theorem, which we now recall in a suitably general, quantitative form due to Cohen \cite{Coh2} (see also \cite[Ch. 9]{Ser}).  Throughout, $k$ will denote a number field.
\begin{theorem}[Hilbert Irreducibility Theorem]
\label{HIT2}
Let $f\in k[x,t_1,\ldots,t_n]$ be a polynomial irreducible in $k[x,t_1,\ldots,t_n]$.  Let $g\in k[t_1,\ldots,t_n]$ be nonzero.  Then for all but $O\left(X^{n-\frac{1}{2}}\log X\right)$ points $(i_1,\ldots, i_n)\in [-X,X]^n\cap \mathbb{Z}^n$, the polynomial $f(x,i_1,\ldots, i_n)$ is irreducible over $k$ and $g(i_1,\ldots, i_n)\neq 0$.  If  $n=1$, $O\left(\sqrt{X}\log X\right)$ can be replaced by $O\left(\sqrt{X}\right)$.
\end{theorem}
We will also find it convenient to use an equivalent geometric formulation of Hilbert's theorem.  Before stating this, we give some definitions and notation.  By a variety over a field $k$, we mean a geometrically integral, separated scheme of finite type over $k$.  A curve is a one-dimensional variety.  We define a dominant rational map $\phi:V\to W$ of varieties over $k$ to be generically finite if the extension of function fields $k(V)/k(W)$, induced by $\phi$, is a finite extension.  For a point ${\bf i}=(i_1,\ldots, i_n)\in \mathbb{A}^n(\mathbb{Z})$, we define $H({\bf i})=\max_j |i_j|$.
\begin{theorem}[Geometric Hilbert Irreducibility Theorem]
\label{HIT}
Let $V$ be a variety over $k$ of dimension $n$.  Let $\phi:V\to \mathbb{A}^n$ be a generically finite rational map.  Let $Z$ be a proper closed subset of $\mathbb{A}^n$.  If ${\bf i}\in \mathbb{A}^n(\mathbb{Z})$ and ${\bf i}\in \phi(V)$, let $P_{\bf i}\in \phi^{-1}({\bf i})$.  Then for all but $O\left(X^{n-\frac{1}{2}}\log X\right)$ points ${\bf i}\in \mathbb{A}^n(\mathbb{Z})$ with $H({\bf i})\leq X$, $P_{\bf i}$ is defined, ${\bf i}\not\in Z$, and $[k(P_{\bf i}):k]=\deg \phi$.   If  $n=1$, $O\left(\sqrt{X}\log X\right)$ can be replaced by $O\left(\sqrt{X}\right)$.
\end{theorem}
For more general and precise statements of Hilbert's Irreducibility Theorem, see \cite[Ch. 9]{Ser} and \cite{Coh2}.
\begin{remark}
\label{Hrem}
Schinzel \cite{Schi} has shown that there exist arithmetic progressions $I_1,\ldots, I_n$ such that the conclusions of Theorems \ref{HIT2} and \ref{HIT} hold, without exception, for arbitrarily chosen elements $i_j\in I_j$, $j=1,\ldots, n$, where ${\bf i}=(i_1,\ldots, i_n)$.
\end{remark}
\begin{remark}
\label{effrem}
Cohen \cite{Coh2} has shown, furthermore, that the implied constants in Theorems \ref{HIT2} and \ref{HIT} can be effectively determined, as well as the arithmetic progressions in Remark \ref{Hrem}.  
\end{remark}
A very natural problem that arises in connection with Hilbert's Irreducibility Theorem is to determine the number of distinct number fields $\mathbb{Q}(P_{{\bf i}})$ with $H({\bf i})\leq X$.  For $n=1$ this has been studied by Dvornicich and Zannier.  Let $C$ be a curve and let $\phi:C\to \mathbb{P}^1$ be a morphism.  For each integer $i\in \mathbb{A}^1(\mathbb{Z})\subset \mathbb{P}^1(\mathbb{Q})$, let $P_i\in \phi^{-1}(i)$.  Dvornicich and Zannier \cite{Zan, Zan2} studied the degree of the field extension $\mathbb{Q}(P_1,\ldots,P_N)$.  Their results imply in particular a useful result on the number of isomorphism classes of number fields in the set $\{\mathbb{Q}(P_1),\ldots,\mathbb{Q}(P_N)\}$.
\begin{theorem}[Dvornicich, Zannier]
\label{Zan}
Let $C$ be a curve over a number field $k$.  Let $\phi:C\to\mathbb{P}^1$ be a morphism with $\deg \phi>1$.  For each integer $i$, let $P_i\in \phi^{-1}(i)$.  Let $g(N)$ denote the number of isomorphism classes of number fields in the set $\{\mathbb{Q}(P_1),\ldots,\mathbb{Q}(P_N)\}$.  Then $g(N)\gg \frac{N}{\log N}$.
\end{theorem}
An analysis of the proof in \cite{Zan} shows that furthermore the implied constant in Theorem \ref{Zan} is effective (for this one makes use of effective versions of Hilbert's Irreducibility Theorem (Remark \ref{effrem}) and the prime number theorem).

\section{Theorems and Proofs}

Before stating our main theorem, we introduce a little more notation.  If $S$ is a set of places of a number field $k$, we let $S_{\rm{fin}}$ denote the set of finite places in $S$.  If $L$ is a finite extension of $k$ and $S_L$ is the set of places of $L$ lying above places of $S$, we will use $\co_{L,S}$ to mean $\co_{L,S_L}$, the ring of $S_L$-integers of $L$.  If $G$ is an abelian group, we let $\rk G$ denote the free rank of $G$ as a $\mathbb{Z}$-module.

\begin{theorem}
\label{thmain}
Let $V$ be a variety over $\mathbb{Q}$ of dimension $l$.  Let $\phi:V\to \mathbb{A}^l$ be a generically finite rational map with $\deg \phi = n>1$.  Let ${\bf I}\subset \phi(V)$ be a subset of $\mathbb{A}^l(\mathbb{Z})$ and let ${\bf I}(N)=\{{\bf i}\in {\bf I}\mid H({\bf i})<N\}$.  For ${\bf i}\in {\bf I}$, let $P_{\bf i}\in \phi^{-1}({\bf i})$.  Let $m>1$ be a positive integer.  Let $f_1,\ldots, f_r\in \mathbb{Q}(V)$ be rational functions on $V$ such that
\begin{equation}
\label{cond1}
\left[\Qbar(V)\left(\sqrt[m]{f_1},\ldots,\sqrt[m]{f_r}\right):\Qbar(V)\right]=m^r.
\end{equation}
Let $S$ be a finite set of places of $\mathbb{Q}$ such that for all ${\bf i}\in {\bf I}$ and all $j$,
\begin{equation}
\label{cond2}
f_j(P_{\bf i})\co_{\mathbb{Q}(P_{\bf i}),S}=\mathfrak{a}_{{\bf i},j}^m
\end{equation}
for some fractional $\co_{\mathbb{Q}(P),S}$-ideal $\mathfrak{a}_{{\bf i},j}$.  
Then for all but $O(N^{l-\frac{1}{2}}\log N)$ (or $O(\sqrt{N})$ if $l=1$) points ${\bf i}\in {\bf I}(N)$, we have $[\mathbb{Q}(P_{\bf i}):\mathbb{Q}]=n$ and
\begin{equation}
\label{maineq2}
\rk_{m} \Cl(\mathbb{Q}(P_{\bf i}))\geq r +\#S_{\rm{fin}}-\rk \co_{\mathbb{Q}(P_{\bf i}),S}^*.
\end{equation}
\end{theorem}

\begin{proof}
Associated to the field extension $\mathbb{Q}(V)\left(\sqrt[m]{f_1},\ldots,\sqrt[m]{f_r}\right)$ of $\mathbb{Q}(V)$, we have a projective variety $X$ over $\mathbb{Q}$ (unique up to birational equivalence) and a morphism $\pi:X\to V$ with $\deg \pi=m^r$. 

\begin{lemma}
\label{ldeg}
Let $Z\subset V$ be the set of poles of the functions $f_1,\ldots, f_r$.  Let $P_{\bf i}\in V(\Qbar)\setminus Z$, for some ${\bf i}\in {\bf I}$.  Let $\zeta$ be a generator for the group of roots of unity in $k=\mathbb{Q}(P_{\bf i})$.  Let $L=\mathbb{Q}(\{\sqrt[m]{q}\mid q\in S_{\rm{fin}}\})$.  Let $Q_{\bf i}\in \pi^{-1}(P_{\bf i})$.  Then for some prime $p$ dividing $m$,
\begin{equation}
\label{degineq}
[L(Q_{\bf i}):L]\leq \left[L\left(\sqrt[p]{\zeta}\right):L\right][k:\mathbb{Q}]m^r p^{\rk_m \Cl(k)+\rk \co_{k,S}^*-r-\#S_{\rm{fin}}}.
\end{equation}
\end{lemma}

\begin{proof}
We actually prove a stronger inequality, with $\rk_m \Cl(k)$ replaced by $\rk_m \Cl(\co_{k,S})$ (the class group of $\co_{k,S}$) in \eqref{degineq}.  We will work throughout with (fractional) $\co_{k,S}$-ideals.  Let $\ord_p m$ denote the largest power of $p$ dividing $m$.  Let $p$ be a prime dividing $m$ such that $\rk_{p^{\ord_p m}}\Cl(\co_{k,S})=\rk_m \Cl(\co_{k,S})=t$.  Let $G=\left\{[\mathfrak{a}]^\frac{m}{p}\mid [\mathfrak{a}]\in\Cl(\co_{k,S}), [\mathfrak{a}]^m=1 \right\}$, a subgroup of $\Cl(\co_{k,S})$.  Clearly, $G\cong (\mathbb{Z}/p\mathbb{Z})^{t}$.  Let $\mathfrak{b}_j$, $j=1,\ldots, t$, be ($\co_{k,S}$-)ideals whose ideal classes generate $G$.  Then for each $j$, $\mathfrak{b}_j^p=(\beta_j)$ for some $\beta_j\in k$.  Let $t'=\rk \co_{k,S}^*$.  Let $u_1,\ldots, u_{t'},\zeta$ be generators for $\co_{k,S}^*$.  Let $M=k(\sqrt[p]{\beta_1},\ldots,\sqrt[p]{\beta_{t}},\sqrt[p]{u_1},\ldots,\sqrt[p]{u_{t'}},\sqrt[p]{\zeta})$.  Let $L'=\mathbb{Q}(\{\sqrt[p]{q}\mid q\in S_{\rm{fin}}\})$.  Note that $L'\subset L\cap M$, $[L':\mathbb{Q}]=p^{\#S_{\rm{fin}}}$, and 
\begin{equation*}
[LM:L]\leq [L(\sqrt[p]{\zeta}):L][k:\mathbb{Q}]p^{t+t'}/[L':\mathbb{Q}]=[L(\sqrt[p]{\zeta}):L][k:\mathbb{Q}]p^{t+t'-\#S_{\rm{fin}}}.
\end{equation*}
From the definitions of $\pi$ and $X$, $k(Q_{\bf i})=k(x_1,\ldots, x_r)$ for some choice of $x_j$ satisfying $x_j^m=f_j(P_{\bf i})$, $j=1,\ldots, r$.  It suffices to show that $x_j^{\frac{m}{p}}\in M$ for all $j$.  Indeed, this gives immediately that 
\begin{equation*}
[L(Q_{\bf i}):L]\leq \left(\frac{m}{p}\right)^r [LM:L]\leq \left[L\left(\sqrt[p]{\zeta}\right):L\right][k:\mathbb{Q}]m^r p^{t+t'-r-\#S_{\rm{fin}}}.
\end{equation*}
By assumption, $(x_j^m)=(f_j(P_{\bf i}))=\mathfrak{a}_j^m$ for some $\co_{k,S}$-ideal $\mathfrak{a}_j$.  Since $[\mathfrak{a}_j]^{\frac{m}{p}}\in G$, 
\begin{equation*}
\mathfrak{a}_j^{\frac{m}{p}}=(\alpha)\prod_{s=1}^{t}\mathfrak{b}_s^{c_s}
\end{equation*}
for some integers $c_s$ and some element $\alpha\in k$.  Therefore,
\begin{equation*}
(x_j^m)=\left(\mathfrak{a}_j^\frac{m}{p}\right)^p=(\alpha^p)\prod_{s=1}^{t}\left(\beta_s^{c_s}\right).
\end{equation*}
So $x_j^m=u\alpha^p\prod_{s=1}^{t}\beta_s^{c_s}$ for some unit $u\in \co_{k,S}^*$. Therefore,  $x_j^{\frac{m}{p}}=\alpha\sqrt[p]{u}\prod_{s=1}^{t}\sqrt[p]{\beta_s^{c_s}}$ for some choice of the $p$-th roots.  So $x_j^{\frac{m}{p}}\in M$ for all $j$ as desired.
\end{proof}

For ${\bf i}\in {\bf I}$, let $Q_{\bf i}\in \pi^{-1}(P_{\bf i})$.  So $\phi(\pi(Q_{\bf i}))={\bf i}\in \mathbb{A}^l(\mathbb{Z})$.  Let $L$ be the number field from Lemma \ref{ldeg}.  Let $M$ be the number field generated over $L$ by elements $\sqrt[m]{\zeta}$, where $\zeta$ is any root of unity appearing in a number field $k$ with $[k:\mathbb{Q}]\leq n$.  Let $W$ be a proper closed subset of $\mathbb{A}^l$ containing $\phi(Z)$.  By \eqref{degineq}, for all ${\bf i}\in {\bf I}\setminus W$,
\begin{equation*}
[M(Q_{\bf i}):M]\leq n m^r p^{\rk_m \Cl(\mathbb{Q}(P_{\bf i}))+\rk \co_{\mathbb{Q}(P_{\bf i}),S}^*-r-\#S_{\rm{fin}}}.
\end{equation*}
On the other hand, $\deg \phi\circ\pi=nm^r$.  Therefore, by Hilbert's Irreducibility Theorem over $M$, we obtain that for all but $O\left(N^{l-\frac{1}{2}}\log N\right)$ (or $O(\sqrt{N})$ if $l=1$) points ${\bf i}\in {\bf I}(N)$, we have $[\mathbb{Q}(P_{\bf i}):\mathbb{Q}]=n$ and $\rk_m \Cl(\mathbb{Q}(P_{\bf i}))\geq r+\#S_{\rm{fin}}-\rk \co_{\mathbb{Q}(P_{\bf i}),S}^*$.
\end{proof}

We will frequently apply Theorem \ref{thmain} to sets ${\bf I}$ and $S$ such that for every ${\bf i}\in {\bf I}$, there exists exactly one place of $\mathbb{Q}(P_{\bf i})$ above each finite place of $S$ (in particular, this holds if $S=\{\infty\}$, the archimedean place of $\mathbb{Q}$).  Note that in this case, \eqref{maineq2} simplifies to
\begin{equation*}
\rk_{m} \Cl(\mathbb{Q}(P_{\bf i}))\geq r -\rk \co_{\mathbb{Q}(P_{\bf i})}^*.
\end{equation*}

The most natural examples of functions $f_1,\ldots, f_r$ satisfying \eqref{cond1} and \eqref{cond2} arise from torsion in the Picard group (group of isomorphism classes of invertible sheaves) of a variety.  If $V$ is a nonsingular  variety over $\mathbb{Q}$, we can identify $\Pic(V)$ with the group of ($\mathbb{Q}$-rational) divisors on $V$ modulo linear equivalence.  We denote the linear equivalence class of a divisor $D$ by $[D]$.  If $f$ is a rational function on $V$, we let $(f)$ be the principal divisor associated to $f$.

\begin{corollary}
\label{cormain}
Let $V$ be a nonsingular projective variety over $\mathbb{Q}$ of dimension $l$.  Let $\phi:V\to \mathbb{A}^l$ be a generically finite rational map with $\deg \phi = n>1$.  Let ${\bf I}=\{P\in \mathbb{A}^l(\mathbb{Z})\mid P\in\phi(V)\}$ and let ${\bf I}(N)=\{{\bf i}\in {\bf I}\mid H({\bf i})<N\}$.  For ${\bf i}\in {\bf I}$, let $P_{\bf i}\in \phi^{-1}({\bf i})$.  Let $m>1$ be a positive integer.  
Then there exists a finite set of places $S$ of $\mathbb{Q}$, depending only on $V$ and $m$ and not $\phi$, such that for all but $O(N^{l-\frac{1}{2}}\log N)$ (or $O(\sqrt{N})$ if $l=1$) points ${\bf i}\in {\bf I}(N)$, we have $[\mathbb{Q}(P_{\bf i}):\mathbb{Q}]=n$ and
\begin{equation}
\label{coreq}
\rk_{m} \Cl(\mathbb{Q}(P_{\bf i}))\geq \rk_m \Pic(V)_{\rm{tors}} +\#S_{\rm{fin}}-\rk \co_{\mathbb{Q}(P_{\bf i}),S}^*.
\end{equation}
\end{corollary}

\begin{proof}
Let $r=\rk_m \Pic(V)_{\rm{tors}}$.  We first show that \eqref{cond1} holds for certain functions $f_1,\ldots, f_r$.
\begin{lemma}
\label{LKum}
Let $D_1,\ldots, D_r$ be divisors whose divisor classes generate a subgroup $(\mathbb{Z}/m\mathbb{Z})^r\subset\Pic(V)$.  Let $f_1,\ldots, f_r\in \mathbb{Q}(V)$ be rational functions such that $(f_j)=mD_j$ for all $j$.  Then 
\begin{equation*}
\left[\Qbar(V)\left(\sqrt[m]{f_1},\ldots,\sqrt[m]{f_r}\right):\Qbar(V)\right]=m^r.
\end{equation*}
\end{lemma}
\begin{proof}
By Kummer theory, this is equivalent to showing that $f_1,\ldots, f_r$ generate a subgroup of cardinality $m^r$ in $\Qbar(V)^*/\Qbar(V)^{m*}$.  Suppose that
\begin{equation}
\label{Kum} 
f_1^{i_1}f_2^{i_2}\cdots f_r^{i_r}=g^m
\end{equation}
for some $g\in \Qbar(V)^*$ and  integers $0\leq i_1,\ldots, i_r<m$.  Let $(g)=E$, the principal divisor associated to $g$.  Then by \eqref{Kum}, $mE=\sum_{j=1}^r mi_jD_j$.  So $E=\sum_{j=1}^r i_jD_j$ is a principal divisor.  Since $[D_1],\ldots, [D_r]$ are independent in $\Pic(V)[m]$ (over $\mathbb{Z}/m\mathbb{Z}$), it follows that $i_j=0$ for all $j$.
\end{proof}
Let $f_1,\ldots, f_r$ be as in Lemma \ref{LKum}.  Since $(f_j)=mD_j$ is an ``$m$-th power" in the divisor group of $V$, the values of $f_j$ must be an $m$-th power outside of some finite set of places $S$ depending only on $f_j$.
\begin{lemma}
\label{fi2}
There exists a finite set of places $S$ of \ $\mathbb{Q}$ such that for all $P\in C(\Qbar)$ and all $j$,
\begin{equation*}
f_j(P)\co_{\mathbb{Q}(P),S}=\mathfrak{a}_{P,j}^m
\end{equation*}
for some fractional $\co_{\mathbb{Q}(P),S}$-ideal $\mathfrak{a}_{P,j}$.
\end{lemma}
\begin{proof}
This is an immediate consequence of Weil's Decomposition Theorem \cite[2.7.15]{Bom2}.
\end{proof}
Thus, \eqref{cond2} holds for a finite set of places $S$ of $\mathbb{Q}$ independent of $\phi$.  We are now done by Theorem \ref{thmain}.
\end{proof}

The main content of Corollary \ref{cormain} is that the set of places $S$ can be chosen independent of $\phi$.  This fact is sufficient to prove the following result for superelliptic curves.

\begin{corollary}
\label{super}
Let $C$ be the normalization of a plane curve defined by 
\begin{equation*}
y^n=f(x), \quad f(x)\in \mathbb{Q}[x], \quad n>1, 
\end{equation*}
with $(\deg f, n)=1$.  Let $m>1$ be an integer.  Then there exist $\gg X^{\frac{1}{(n-1)\deg f}}/\log X$ number fields $k$ of degree $[k:\mathbb{Q}]=n$ with discriminant $d_k$, $|d_k|<X$, and
\begin{equation*}
\rk_m \Cl(k)\geq \rk_m \Jac(C)(\mathbb{Q})_{\rm{tors}}-\left[\frac{n-1}{2}\right].
\end{equation*}
\end{corollary}
\begin{proof}
Let $t=\deg f$ and let $f=\sum_{i=0}^t a_ix^i$.  By rescaling $x$ and $y$ we can assume that $a_i\in \mathbb{Z}$ for all $i$ and, using $(t,n)=1$, that $a_t=-1$.  Let $S$ be as in Corollary \ref{cormain} (for $V=C$ and $m$).  Let $M=\prod_{p\in S_{\rm{fin}}}p$.  Let $\psi$ be the rational function on $y^n=f(x)$ defined by $\psi=x-\frac{1}{M}$.  For  $i\in \mathbb{Z}$, let $P_i\in \psi^{-1}(i)$.  Then $\mathbb{Q}(P_i)\cong \mathbb{Q}\left(\sqrt[n]{f\left(\frac{iM+1}{M}\right)}\right)$.  Let $p\in S_{\rm{fin}}$.  Since $a_t=-1$, we have that $\ord_p f\left(\frac{iM+1}{M}\right)=-t$.  This implies that $p$ totally ramifies in $\mathbb{Q}(P_i)$.  Explicitly, if $i,j\in \mathbb{N}$ is such that $ni-tj=1$, then $\left(M^i\sqrt[n]{f\left(\frac{iM+1}{M}\right)^j},p\right)^n=(p)$ in $\co_{\mathbb{Q}(P_i)}$.  Note also that $[\mathbb{Q}(P_i):\mathbb{Q}]=n$.  So every prime of $S_{\rm{fin}}$ is totally ramified in $\mathbb{Q}(P_i)$.  Let $\phi:C\to \mathbb{P}^1$ be the morphism induced by $\psi$.  The condition $(t,n)=1$ implies that $y^n=f(x)$ has a rational point at infinity.  Since $C$ has a rational point, $\Pic(C)_{\rm{tors}}\cong \Jac(C)(\mathbb{Q})_{\rm{tors}}$ (this follows, for instance, from Lemma \ref{fi}).  Corollary \ref{cormain} applied to $\phi$ and $C$ then implies that for all but $O(\sqrt{N})$ values $i=1,\ldots, N$, we have  $[\mathbb{Q}(P_i):\mathbb{Q}]=n$ and
\begin{equation*}
\rk_m \Cl(\mathbb{Q}(P_i))\geq \rk_m \Jac(C)(\mathbb{Q})_{\rm{tors}}-\rk \co_{\mathbb{Q}(P_i)}^*.
\end{equation*}
Since $a_t=-1$, for $i\gg 0$, $f\left(\frac{iM+1}{M}\right)$ is negative.  It follows that for $i\gg 0$, $\mathbb{Q}(P_i)$ has exactly one real place if $n$ is odd and no real places if $n$ is even.  So by Dirichlet's theorem, $\rk \co_{\mathbb{Q}(P_i)}^*=\left[\frac{n-1}{2}\right]$ for $i\gg 0$.  It follows from Theorem \ref{Zan} that there are $\gg N/\log N$ distinct number fields in the set $\{\mathbb{Q}(P_1),\ldots,\mathbb{Q}(P_N)\}$.  An easy calculation shows that $|d_{\mathbb{Q}(P_i)}|=O(i^{(n-1)t})$.  Combining the above statements then gives the corollary.
\end{proof}

Of particular interest is the case where $C$ is a hyperelliptic curve.

\begin{corollary}
\label{chyp}
Let $C$ be a nonsingular hyperelliptic curve over $\mathbb{Q}$ with a rational Weierstrass point.  Let $g$ denote the genus of $C$.  Let $m>1$ be an integer.  Then there exist $\gg X^{\frac{1}{2g+1}}/\log X$ imaginary quadratic number fields $k$ with
\begin{equation*}
\rk_m \Cl(k)\geq \rk_m \Jac(C)(\mathbb{Q})_{\rm{tors}},\quad |d_k|<X,
\end{equation*}
and $\gg X^{\frac{1}{2g+1}}/\log X$ real quadratic number fields $k$ with
\begin{equation*}
\rk_m \Cl(k)\geq \rk_m \Jac(C)(\mathbb{Q})_{\rm{tors}}-1,\quad d_k<X.
\end{equation*}
\end{corollary}
\begin{proof}
Since $C$ has a rational Weierstrass point, $C$ is birational to a plane curve defined by $y^2=f(x)$, for some $f(x)\in \mathbb{Q}[x]$ with $\deg f=2g+1$.  The statement on imaginary quadratic fields now follows from Corollary \ref{super}.  The statement for real quadratic fields follows similarly from the proof of Corollary \ref{super} using points $P_i$, $i<0$, with the extra $-1$ term coming from the rank-one unit group of a real quadratic field.
\end{proof}

As another variation on Corollary \ref{cormain}, we prove a theorem for curves $C$ using $\Jac(C)(\mathbb{Q})_{\rm{tors}}$ rather than $\Pic(C)_{\rm{tors}}$, and we give a simple explicit description of a viable set of places $S$ in \eqref{coreq}.  We let $\bar{C}$ be the curve $C$ base extended to $\Qbar$, and we naturally identify objects on $C$ (rational functions, divisors, etc.) with the corresponding objects on $\bar{C}$. 
\begin{corollary}
\label{thmain2}
Let $C$ be a nonsingular projective curve over $\mathbb{Q}$.  Let $\phi:C\to \mathbb{P}^1$ be a morphism with $\deg \phi = n>1$.  Let $m>1$ be a positive integer.  Let 
\begin{equation*}
d=\gcd\{[k:\mathbb{Q}]\mid \bar{C} \text{ has a $k$-rational degree one divisor}\}.
\end{equation*}
Let $m'=\frac{m}{(m,d)}$ with $m'>1$.  Let $T$ be the union of the set of primes dividing $m$, the set of primes of bad reduction of $\Jac(C)$, and the archimedean prime of $\mathbb{Q}$.  For $i\in \mathbb{Z}$, let $P_i\in \phi^{-1}(i)$.  Then for all but $O(\sqrt{N})$ values $i=1,\ldots, N$, we have $[\mathbb{Q}(P_i):\mathbb{Q}]=n$ and
\begin{equation*}
\rk_{m'} \Cl(\mathbb{Q}(P_i))\geq \rk_m \Jac(C)(\mathbb{Q})_{\rm{tors}}+\#T_{\rm{fin}}-\rk \co_{\mathbb{Q}(P_i),T}^*.
\end{equation*}
\end{corollary}

For simplicity, we prove Corollary \ref{thmain2} in the case $m=m'$, or equivalently, $(d,m)=1$.  The proof of the general case of the corollary requires only minor modifications.  Let $r=\rk_m \Jac(C)(\mathbb{Q})_{\rm{tors}}$.  Corollary \ref{thmain2} follows from Theorem \ref{thmain}, Lemma \ref{LKum}, and the following lemma.
\begin{lemma}
\label{fi}
There exist $\mathbb{Q}$-rational divisors $D_1,\ldots, D_r$ whose divisor classes generate a subgroup $(\mathbb{Z}/m\mathbb{Z})^r\subset\Jac(C)(\mathbb{Q})$ and rational functions $f_j\in \mathbb{Q}(C)$, $j=1,\ldots, r$, such that $(f_j)=mD_j$ and for all but $O(\sqrt{N})$ values $i=1,\ldots, N$,
\begin{equation}
\label{power}
f_j(P_i)\co_{\mathbb{Q}(P_i),T}=\mathfrak{a}_{i,j}^m
\end{equation}
for some fractional $\co_{\mathbb{Q}(P_i),T}$-ideal $\mathfrak{a}_{i,j}$.
\end{lemma}
\begin{proof}
Let $c_1,\ldots, c_r$ be generators for a subgroup $(\mathbb{Z}/m\mathbb{Z})^r\subset\Jac(C)(\mathbb{Q})$.  We first claim that it suffices to show that if $\bar{C}$ has a degree one $k$-rational divisor $D$ and $[k:\mathbb{Q}]=e$, then there exist $\mathbb{Q}$-rational divisors $D_1,\ldots, D_r$ and rational functions $f_j\in \mathbb{Q}(C)$, $j=1,\ldots, r$, such that $[D_j]=ec_j$, $(f_j)=mD_j$, and \eqref{power} holds.  This can be proven as follows.  Since $(m,d)=1$, there exist number fields $k_l$, $[k_l:\mathbb{Q}]=e_l$, $l=1,\ldots, L$, such that $\sum_{l=1}^L e_l\equiv 1 \mod m$ and for each $l$ there exists a degree one $k_l$-rational divisor on $\bar{C}$.  Let $D_{j,l}$ and $f_{j,l}$ be divisors and rational functions, as in our claim above, for $k_l$.  Let $D_j=\sum_{l=1}^LD_{j,l}$ and $f_j=\prod_{l=1}^L f_{j,l}$.  Then using that $\sum_{l=1}^L e_l\equiv 1 \mod m$, we get that $[D_j]=c_j$ and it is easily seen that $f_j$ and $D_j$ satisfy the conditions of the theorem.

Let $D$ be a degree one $k$-rational divisor on $\bar{C}$.  Let $[k:\mathbb{Q}]=e$.  Let $\psi:\bar{C}\hookrightarrow J=\Jac(\bar{C})$ be the $k$-rational embedding given by $P\mapsto [P-D]$.  Let $\Theta=\psi(\bar{C})+\ldots+\psi(\bar{C})$ be the theta divisor on $J$.  Let $E_j=\Theta-t_{c_j}^*\Theta$, where $t_{c_j}$ denotes the translation-by-$c_j$ map on $J$.  By the theorem of the square, $mE_j$ is a principal divisor.  Let $g_j \in k(J)$ be such that $(g_j)=mE_j$.  Since $[m]^*E_j\sim mE_j$ is principal, where $[m]$ is multiplication by $m$ on $J$, let $h_j\in k(J)$ be such that $(h_j)=[m]^*E_j$.  It follows immediately that $g_j(mx)=\alpha_jh_j(x)^m$ for some constant $\alpha_j\in k^*$.  Replacing $g_j$ by $g_j/\alpha_j$, we can assume that $g_j(mx)=h_j(x)^m$.  Let $x,y\in J(\Qbar)$ with $my=x$.  It is a standard fact that the extension $k(y)/k(x)$ is unramified outside of (places lying above) $T$.  Since $g_j(x)=g_j(my)=h_j(y)^m$ and $k(y)/k(x)$ is unramified outside of $T$, it follows that  $g_j(x)\co_{k(x),T}=\mathfrak{a}_j^m$ for some fractional $\co_{k(x),T}$-ideal $\mathfrak{a}_j$.  Let $\Nk$ denote any map naturally induced by the norm map from $k$ to $\mathbb{Q}$.  Let $f_j=\Nk g_j|_{\bar{C}}$, via the embedding $\psi:\bar{C}\hookrightarrow J$.  Let $D_j=\Nk \psi^*(\Theta-t_{c_j}^*\Theta)$.  Then $f_j$ and $D_j$ are both defined over $\mathbb{Q}$.  Furthermore, $(f_j)=\Nk \psi^*(mE_j)=m\Nk \psi^*(\Theta-t_{c_j}^*\Theta)=mD_j$.  Since $[\psi^*(\Theta-t_{c_j}^*\Theta)]=c_j$ \cite[Th.\ A.8.2.1]{SH}, $[D_j]=\Nk c_j=ec_j$.  By Hilbert's Irreducibility Theorem, for all but $O(\sqrt{N})$ values $i=1,\ldots, N$, $[k(P_i):k]=\deg \phi=n$.  For such values of $i$, $f_j(P_i)=\left(\Nk g_j|_{\bar{C}}\right)(P_i)=\Nki \left(g_j|_{\bar{C}}(P_i)\right)$.  Thus, from the corresponding property for $g_j$, we obtain that \eqref{power} holds for $f_j$ and all but $O(\sqrt{N})$ values $i=1,\ldots, N$.
\end{proof}

\section{Applications and Examples}
\label{apps}
Using the theorems of the last section, we can rederive, in a simple manner, many of the results in Table \ref{table1}.  Additionally, in some cases we are able to obtain new quantitative results.
\subsection{Number fields $k$ of degree $n$ with $\rk_n \Cl(k)$ arbitrarily large}
\label{Rosen}
We give a simple proof of the fact that there exist number fields $k$ of degree $n$ with $\rk_n \Cl(k)$ arbitrarily large.  For $n=2$, this follows, as mentioned earlier, from a much more precise result of Gauss.  For arbitrary $n$, this is implicit in \cite{Bru} (or possibly \cite{Ros}).  The following theorem and proof are essentially taken from \cite{Lev2}.
\begin{theorem}
\label{rn}
Let $n>1$ be a positive integer.  Let $f(x)=\prod_{i=1}^r(a_ix-b_i)$, $a_i, b_i\in \mathbb{Z}$ for all $i$, be a polynomial with $r$ distinct roots.  Let 
\begin{equation*}
T=\{x\in \mathbb{Z}\mid (a_ix-b_i,a_jx-b_j)=1, \forall i,j, i\neq j\}.
\end{equation*}
Then for all but $O(\sqrt{N})$ values $x\in T\cap [-N,\ldots,N]$,
\begin{equation*}
\rk_n \Cl\left(\mathbb{Q}\left(\sqrt[n]{f(x)}\right)\right)\geq r-\rk \co_{\mathbb{Q}\left(\sqrt[n]{f(x)}\right)}^*-1.
\end{equation*}
\end{theorem}
\begin{proof}
Let $C$ be the curve defined by $y^n=f(x)$.  Let $\phi:C\to \mathbb{A}^1$ be given by $(x,y)\mapsto x$.  Let $x_0\in T$ and let $y_0=\sqrt[n]{f(x_0)}$.  Let $k=\mathbb{Q}(y_0)$.  From the definition of $T$ and the fact that $y_0^n=f(x_0)$, it follows that for every $j$, $(a_jx_0-b_j)=\mathfrak{a}_j^n$ for some ideal $\mathfrak{a}_j$ of $\co_k$.  Furthermore, 
\begin{equation*}
\left[\Qbar(C)\left(\sqrt[n]{a_1x-b_1},\ldots,\sqrt[n]{a_{r-1}x-b_{r-1}}\right):\Qbar(C)\right]=n^{r-1}
\end{equation*}
is equivalent to the assertion that
\begin{equation*}
(a_1x-b_1)^{i_1}\cdots(a_{r-1}x-b_{r-1})^{i_{r-1}}=f(x)^{i_r}g(x)^n, \quad g\in \Qbar(x), 0\leq i_1,\ldots, i_{r}<n
\end{equation*}
has only the trivial solution with $i_1=\cdots=i_{r}=0$. This is obviously true under our assumptions.
We now apply Theorem \ref{thmain} (with $V=C$, $\phi$, ${\bf I}=T$, $S=\{\infty\}$, and $f_j=a_jx-b_j$, $j=1,\ldots, r-1$).  When $(n,r)=1$, we note that if $\tilde{C}$ is a nonsingular projective model of $C$, then the divisors $D_1,\ldots, D_r$ on $\tilde{C}$ with $(f_j)=nD_j$ give rise to a subgroup $(\mathbb{Z}/n\mathbb{Z})^{r-1}$ in $\Jac(\tilde{C})(\mathbb{Q})$ (see Corollary~\ref{super}).
\end{proof}
The set $T$ in Theorem \ref{rn} may be infinite or empty depending upon $f$.  Since $r$ can be made arbitrary large with respect to $n$, by choosing appropriate polynomials $f$ we obtain the following corollary.
\begin{corollary}[Brumer, Rosen \cite{Bru,Ros}]
Let $n>1$ be a positive integer.  Then for any integer $r_0$ there exist infinitely many number fields $k$ of degree $n$ with $\rk_n \Cl(k)>r_0$.
\end{corollary}

Theorem \ref{rn} fits into a series of general results \cite{Bru,Ros,Sus,Zas,Sch} giving information on $\Cl(k)$ in terms of the ramification behavior of rational primes in $\co_k$.  Indeed, Theorem \ref{rn} is easily derived, at least when $n$ is prime, from the following general result of Roquette and Zassenhaus (see \cite{Sus} and \cite{Sch} for generalizations).
\begin{theorem}[Roquette, Zassenhaus \cite{Zas}]
Let $k$ be a number field of degree $n$.  Let $p$ be a prime dividing $n$.  Let $L=\mathbb{Q}(\{a\in k\mid a^p\in \mathbb{Q}\})$.  Let $r$ be the number of prime ideals of $\mathbb{Q}$ that are $p$-th powers of ideals of $k$.  Then
\begin{equation*}
\rk_p \Cl(k)\geq r-\rk \co_k^*-\ord_p [L:\mathbb{Q}],
\end{equation*}
where $\ord_p m$ denotes the largest power of $p$ dividing $m$.
\end{theorem}

\subsection{Imaginary quadratic fields with an element of order $m$ in the class group}

Let $k$ be an imaginary quadratic field with an integral ideal $\mathfrak{a}$ of order $m$ in $\Cl(k)$.  Then $\mathfrak{a}^m=(\alpha)$ for some $\alpha\in \co_k$.  Taking norms, we obtain an equation 
\begin{equation*}
4x^m=y^2+t^2d
\end{equation*}
in integers $x$, $y$, $t$, and $d$, where $d$ is squarefree, $k=\mathbb{Q}\left(\sqrt{-d}\right)$, $\alpha=\frac{y+t\sqrt{-d}}{2}$, and $\NN(\mathfrak{a})=x$.  The next theorem gives a sort of converse to this fact.
\begin{theorem}
\label{quadcl}
Let $T(N)=\{(x,y)\in \mathbb{Z}^2\mid |x|,|y|<N, y^2-4x^m<0, (x,y)=1\}$.  For all but $O(N^{\frac{3}{2}}\log N)$ elements $(x,y)\in T(N)$, the field $\mathbb{Q}\left(\sqrt{y^2-4x^m}\right)$ has an element of order $m$ in its class group.
\end{theorem}
\begin{proof}
Consider the surface $V$ defined by $z^2=y^2-4x^m$.  Let $x_0,y_0\in \mathbb{Z}$ be such that $z_0^2=y_0^2-4x_0^m<0$ and $(x_0,y_0)=1$.  Let $k=\mathbb{Q}(z_0)$.  Then $\frac{y_0+z_0}{2}, \frac{y_0-z_0}{2}\in \co_k$ and $(\frac{y_0+z_0}{2}, \frac{y_0-z_0}{2})=(x_0,y_0)=\co_k$.  Thus, from the factorization $(\frac{y_0+z_0}{2})(\frac{y_0-z_0}{2})=x_0^m$, we have that $(\frac{y_0+z_0}{2})=\mathfrak{a}^m$ for some ideal $\mathfrak{a}$ of $\co_k$.  Furthermore, it is easy to see that 
\begin{equation*}
\left[\Qbar(V)\left(\sqrt[m]{\frac{y+z}{2}}\right):\Qbar(V)\right]=m.
\end{equation*}
Let $\phi:V\to \mathbb{A}^2$ be the projection $(x,y,z)\mapsto (x,y)$.  Then we apply Theorem \ref{thmain} to $V$, $\phi$, ${\bf I}=T(\infty)$, $S=\{\infty\}$, and $f_1=\frac{y+z}{2}$.
\end{proof}
By Remark \ref{Hrem}, there exist integers $x_0, y_0, N\in \mathbb{Z}$ (which, from the proof in \cite{Coh2}, can also be chosen with $(x_0,y_0,N)=1$) such that if $(x,y)\in T$, $x\equiv x_0\pmod N$, $y\equiv y_0 \pmod N$, then $\mathbb{Q}\left(\sqrt{y^2-4x^m}\right)$ has an element of order $m$ in its class group.  Indeed, Yamamoto explicitly showed this in \cite{Yam}.
\begin{theorem}[Yamamoto]
Let $m>1$ be a positive integer with prime divisors $p_1,\ldots, p_s$.  Let $q_1,\ldots, q_s$ be primes such that
\begin{equation*}
q_i \equiv 
\begin{cases}
1 \pmod{p_i}, \quad \text{ if } p_i\neq 2,\\
 1 \pmod 4, \quad \text{ if } p_i=2.
\end{cases}
\end{equation*}
Let $x$ and $y$ be integers such that $(x,y)=1$, $y^2-4x^m<0$, $q_i$ divides $x$ for all $i$, and $y$ is a $p_i$-th power nonresidue mod $q_i$ for all $i$.  Assume also that $\mathbb{Q}\left(\sqrt{y^2-4x^m}\right)\neq \mathbb{Q}(\sqrt{-1}),\mathbb{Q}(\sqrt{-3})$.  Then $\Cl\left(\mathbb{Q}\left(\sqrt{y^2-4x^m}\right)\right)$ contains an element of order $m$.
\end{theorem}
\subsection{Real quadratic fields with an element of order $m$ in the class group}
It is fairly easy to construct, for any integer $m>1$, hyperelliptic curves $C$ with a rational Weierstrass point and $\rk_m \Jac(C)(\mathbb{Q})_{\rm tors}\geq 2$.  Thus, using Corollary~\ref{chyp}, we obtain quantitative versions of results of Yamamoto and Yamamoto-Weinberger.
\begin{theorem}
\label{ri}
Let $m>1$ be an integer.  There exist $\gg X^{\frac{1}{2m-1}}/\log X$ imaginary quadratic number fields $k$ with $|d_k|<X$ and $\rk_m \Cl(k)\geq 2$ and $\gg X^{\frac{1}{2m-1}}/\log X$ real quadratic number fields $k$ with $d_k<X$ and $\rk_m \Cl(k)\geq 1$.
\end{theorem}
\begin{proof}
We construct hyperelliptic curves $C$ of genus $m-1$ with a rational Weierstrass point and $\rk_m \Jac(C)(\mathbb{Q})_{\rm tors}\geq 2$.  Let $C$ be a projective curve in $\mathbb{P}^2$ defined by $(y-az)(y-bz)z^{m-1}=x^m(y-cz)$ with $a,b,c\in \mathbb{Q}$ all distinct.  This curve has a singularity only at $(0,1,0)$.  Let $\pi:\tilde{C}\to C$ be a normalization.  It is easy to see that $\pi$ is bijective on points.  Let $P_0=\pi^{-1}((0,1,0))$, $P_1=\pi^{-1}((0,a,1))$, and $P_2=\pi^{-1}((1,0,0))$.  Let $D_1=P_1-P_0$ and $D_2=P_2-P_0$.  Then it is easily checked that $(\frac{y-az}{z}\circ \pi)=mD_1$ and $(\frac{y-cz}{z}\circ \pi)=mD_2$.  Suppose that $iD_1+jD_2$ is principal for some $0\leq i,j<m$.  This implies that $(Y-a)^i(Y-c)^j=f^m$ for some $f$ in the function field $\Qbar(X,Y)$, where $X=\frac{x}{z}\circ \pi$ and $Y=\frac{y}{z}\circ \pi$.  This is clearly impossible unless $i=j=0$.  Thus $\rk_m \Jac(\tilde{C})(\mathbb{Q})_{\rm tors}\geq 2$.  Completing the square in $y$, we see that the set of curves $C$ for distinct $a,b,c\in \mathbb{Q}$ is the same as the set of hyperelliptic curves given by $y^2=x^{2m}+dx^m+e^2$, $d,e\in \mathbb{Q}$, $e\neq 0$, $d^2\neq 4e^2$.  Such curves may have rational Weierstrass points (e.g., curves with $d=-1-e^2$), and so the result follows from Corollary \ref{chyp}.
\end{proof}
The exponent $\frac{1}{2m-1}$ in Theorem \ref{ri} is not as good as the exponent $\frac{1}{m}$ (or $\frac{1}{m}-\epsilon$) achieved in the papers of Byeon \cite{Bye3}, Luca \cite{Luca2}, and Yu \cite{Yu}.  However, the case of Theorem \ref{ri} where $k$ is imaginary quadratic and $n$ is even does not appear to have been previously covered, and so here Theorem \ref{ri} does give a new result.  

To obtain their results, Byeon and Yu looked at the binary form $\frac{3}{4}(3x^m+y^m)(x^m+3y^m)$.  In the hope of improving their results (possibly to an exponent of $\frac{3}{2m}-\epsilon$), we prove a theorem involving a ternary form.
\begin{theorem}
Let $m>1$ be a positive integer.  Let 
\begin{equation*}
f(x,y,z)=x^{2m}+y^{2m}+z^{2m}-2x^my^m-2x^mz^m-2y^mz^m.
\end{equation*}
For all but $O(N^{\frac{5}{2}}\log N)$ elements $(x,y,z)\in \mathbb{Z}^3$ with $|x|, |y|, |z|<N$ and $(x^m-y^m,z)=(x^m-z^m,y)=(y^m-z^m,x)=1$,
\begin{equation*}
\rk_m \Cl\left(\mathbb{Q}\left(\sqrt{f(x,y,z)}\right)\right)\geq
\begin{cases}
2 & \text{ if }f(x,y,z)<0,\\
1 & \text{ if }f(x,y,z)>0.
\end{cases}
\end{equation*}
\end{theorem}
\begin{proof}
Consider the threefold $V$ defined by $v^2=f(x,y,z)$ and the two factorizations
\begin{align}
\left(\frac{v+(x^m+y^m-z^m)}{2}\right)\left(\frac{v-(x^m+y^m-z^m)}{2}\right)=-x^my^m,\label{fac1}\\
\left(\frac{v+(x^m-y^m+z^m)}{2}\right)\left(\frac{v-(x^m-y^m+z^m)}{2}\right)=-x^mz^m.\label{fac2}
\end{align}
Let $x_0,y_0,z_0\in \mathbb{Z}$ with $(x_0^m-y_0^m,z_0)=(x_0^m-z_0^m,y_0)=(y_0^m-z_0^m,x_0)=1$.  Let $v_0=\sqrt{f(x_0,y_0,z_0)}$ and $k=\mathbb{Q}(v_0)$.  The coprimality conditions immediately imply that $(v_0,x_0)=(v_0,y_0)=(v_0,z_0)=\co_k$ and
\begin{align*}
\left(\frac{v_0+(x_0^m+y_0^m-z_0^m)}{2},\frac{v_0-(x_0^m+y_0^m-z_0^m)}{2}\right)=\co_k\\
\left(\frac{v_0+(x_0^m-y_0^m+z_0^m)}{2},\frac{v_0-(x_0^m-y_0^m+z_0^m)}{2}\right)=\co_k.
\end{align*}
It follows that $\left(\frac{v_0+(x_0^m+y_0^m-z_0^m)}{2}\right)$ and $\left(\frac{v_0+(x_0^m-y_0^m+z_0^m)}{2}\right)$ are both $m$-th powers of ideals of $\co_k$.  Moreover, by looking at the threefold $(v'+y^m)(v'+z^m)=v'x^m$ birational to $V$, the proof of Theorem \ref{ri} easily translates to a proof that
\begin{equation*}
\left[\Qbar(V)\left(\sqrt[m]{\frac{v+x^m+y^m-z^m}{2}},\sqrt[m]{\frac{v+x^m-y^m+z^m}{2}}\right):\Qbar(V)\right]=m^2.
\end{equation*}
So an appropriate application of Theorem \ref{thmain} gives the desired result.
\end{proof}

\subsection{$3$-ranks of quadratic fields and constructions of Craig}
In \cite{Cra} and \cite{Cra2}, Craig constructed infinitely many imaginary quadratic fields $k$ with $\rk_3\Cl(k)\geq 3$ and $\rk_3\Cl(k)\geq 4$, respectively, and infinitely many real quadratic fields $k$ with $\rk_3\Cl(k)\geq 2$ and $\rk_3\Cl(k)\geq 3$, respectively.  We prove quantitative versions of Craig's theorems and show how his constructions yield hyperelliptic curves $C$ with Jacobians having large rational $3$-torsion subgroups.

We begin with the constructions in \cite{Cra}.  Let
\begin{equation*}
f(x,y,z)=x^{6}+y^{6}+z^{6}-2x^3y^3-2x^3z^3-2y^3z^3,
\end{equation*}
the function of the last section for $m=3$.
Then if $v^2=f(x,y,z)$, we have the two factorizations \eqref{fac1} and \eqref{fac2}.  As noticed by Craig, if 
\begin{equation}
\label{Cr1}
2(x^3+y^3)=z^3+w^3,
\end{equation}
then we get the third factorization
\begin{equation}
\label{fac3}
\left(\frac{v+(x^3-y^3+w^3)}{2}\right)\left(\frac{v-(x^3-y^3+w^3)}{2}\right)=-x^3w^3.
\end{equation}
There are two more such factorizations, but they will not lead to any new ideal classes.  A parametric solution to \eqref{Cr1} is given by
\begin{align}
\label{Cr2}
&x=18s^4, &y=3s(t^3-6s^3),\\
&z=t^4, &w=t(18s^3-t^3).
\label{Cr3}
\end{align}
Let $F(s,t)=f(x(s,t),y(s,t),z(s,t))$.  Explicitly,
\begin{multline*}
F(s,t)= t^{24} - 54 s^3 t^{21}+ 1701 s^6t^{18} - 32076  s^9t^{15} +  393660  s^{12}t^{12}\\
    {}- 3464208  s^{15}t^9 + 19840464  s^{18}t^6 - 68024448 s^{21} t^3 + 136048896 s^{24}.
\end{multline*}
For many values of $s$ and $t$, $\rk_3 \Cl\left(\mathbb{Q}\left(\sqrt{F(s,t)}\right)\right)\geq 3- \rk \co_{\mathbb{Q}\left(\sqrt{F(s,t)}\right)}^*$.
\begin{theorem}
\label{LCra1}
Let 
\begin{equation*}
T=\{(s,t)\in \mathbb{Z}^2\mid (3s,t)=1, (s,2)=2, (s+2^it,7)=1, i=0,1,2\}
\end{equation*}
and let $T(N)=\{(s,t)\in T\mid |s|,|t|<N\}$.  Then for all but $O\left(N^{\frac{3}{2}}\log N\right)$ pairs $(s,t)\in T(N)$,
\begin{equation*}
\rk_3 \Cl\left(\mathbb{Q}\left(\sqrt{F(s,t)}\right)\right)\geq
\begin{cases}
3 & \text{ if }F(s,t)<0,\\
2 & \text{ if }F(s,t)>0.
\end{cases}
\end{equation*}
Furthermore, there are $\gg X^{\frac{1}{24}}/\log X$ imaginary quadratic fields $k$ with $|d_k|<X$ and $\rk_3 \Cl(k)\geq 3$ and $\gg X^{\frac{1}{24}}/\log X$ real quadratic fields  $k$ with $d_k<X$ and $\rk_3 \Cl(k)\geq 2$.
\end{theorem}
\begin{proof}
Let $x=x(s,t)$, $y=y(s,t)$, $z=z(s,t)$, and $w=w(s,t)$ as in \eqref{Cr2} and \eqref{Cr3}.  Let $V$ be the variety defined by $v^2=F(s,t)$.  Then we have the three identities \eqref{fac1}, \eqref{fac2}, and \eqref{fac3}.  Furthermore, the condition that each pair on the left-hand sides of \eqref{fac1}, \eqref{fac2}, and \eqref{fac3} be coprime is easily seen to be equivalent to the conditions $(3s,t)=1$, $(s,2)=2$, and $(s+2^it,7)=1$, $i=0,1,2$ (see \cite[Lemma 0]{Cra}).  Let $g_1=\frac{v+x^3+y^3-z^3}{2}$, $g_2=\frac{v+x^3-y^3+z^3}{2}$, and $g_3=\frac{v+x^3-y^3+w^3}{2}$.  Thus, if $(s,t)\in T$ and $v^2=F(s,t)$, then each of the ideals $(g_1(s,t,v))$, $(g_2(s,t,v))$, and $(g_3(s,t,v))$ is the cube of an ideal in $\co_{\mathbb{Q}(v)}$.  To show the first part of the theorem, it remains to show that
\begin{equation}
\label{Crdeg}
\left[\Qbar(V)\left(\sqrt[3]{g_1},\sqrt[3]{g_2},\sqrt[3]{g_3}\right):\Qbar(V)\right]=27.
\end{equation}
Actually, something stronger is true.  Let $C$ be the (nonsingular projective model of the) hyperelliptic curve defined by $v^2=F(1,t)$.  Letting $s=1$, we have the three rational functions $g_1,g_2,g_3\in \mathbb{Q}(C)$.  It follows easily from \eqref{fac1}, \eqref{fac2}, and \eqref{fac3} that $(g_1)=3D_1$, $(g_2)=3D_2$, and $(g_3)=3D_3$ for some degree zero divisors $D_1$, $D_2$, and $D_3$ on $C$.  Moreover, we claim that $D_1$, $D_2$, and $D_3$ generate a subgroup $(\mathbb{Z}/3\mathbb{Z})^3\subset \Jac(C)(\mathbb{Q})$.  For hyperelliptic curves, there is an efficient algorithm (due to Cantor \cite{Can}) for adding points in $\Jac(C)$.  This has been implemented, for instance, in the computer algebra system Magma \cite{Magma}.  Using Cantor's algorithm, it is easily verified in Magma (taking a few seconds) that $D_1$, $D_2$, and $D_3$ generate a $3$-torsion subgroup of rank three in $\Jac(C)(\mathbb{Q})$, as claimed.  It follows from Lemma  \ref{LKum} that \eqref{Crdeg} holds.  So appropriately applying Theorem \ref{thmain} gives the first part of the theorem.

For the last statement of the theorem, we begin by noting that $F(1,2)=-1228544<0$.  It follows that for all sufficiently large $u$, $F(84u-18,168u-35)<0$ (note that $F$ is homogeneous in $s$ and $t$).  Furthermore, if $s=84u-18$, $t=168u-35$, and $u\in \mathbb{Z}$, then $(3s,t)=1$, $(s,2)=2$, and $(s+2^it,7)=1$, $i=0,1,2$.  Consider the rational function $\phi(t,v)=\frac{18t-35}{84t-168}$ on $C$.  Solving $\phi(t,v)=u$ gives $t=\frac{168u-35}{84u-18}$.  Then it follows from all of the above that by applying Theorem \ref{thmain} to the curve $C$ and the rational functions $\phi$, $g_1(1,t,v)$, $g_2(1,t,v)$, and $g_3(1,t,v)$, we obtain infinitely many imaginary quadratic fields $k$ with $\rk_3 \Cl(k)\geq 3$.  More precisely, using Theorem \ref{Zan}, since $\deg F(1,t)=24$, we obtain $\gg X^{\frac{1}{24}}/\log X$ imaginary quadratic fields $k$ with $|d_k|<X$ and $\rk_3 \Cl(k)\geq 3$.  The proof of the statement for real quadratic fields is even easier. 
\end{proof}
We now examine the constructions of Craig from \cite{Cra2}.  The idea in \cite{Cra2} is to find a nontrivial parametric family of solutions to the equations
\begin{equation*}
f(x_0,y_0,z_0)=f(x_1,y_1,z_1)=f(x_2,y_2,z_2).
\end{equation*}
Since 
\begin{equation*}
f(x,y,z)=(x^3+y^3-z^3)^2-4x^3y^3=(x^3-y^3+z^3)^2-4x^3z^3=(-x^3+y^3+z^3)^2-4y^3z^3,
\end{equation*}
it suffices to find solutions to
\begin{align}
x_1z_1=x_0z_0, \quad x_2y_2=x_0y_0,\label{Cr4}\\
x_1^3-y_1^3+z_1^3=-(x_0^3-y_0^3+z_0^3),\label{Cr5}\\
x_2^3+y_2^3-z_2^3=-(x_0^3+y_0^3-z_0^3).\label{Cr6}
\end{align}
Craig gives a two-parameter family of solutions to \eqref{Cr4}, \eqref{Cr5}, and \eqref{Cr6} in terms of $\alpha$, $\beta$, and $\gamma$ satisfying $\alpha+\beta+\gamma=0$.  We refer the reader to \cite{Cra2} for the rather involved formulas.  We specialize Craig's solution by setting $\alpha=0$, $\beta=t$, and $\gamma=-t$.  This gives a polynomial $h(t)=f(x_0(t),y_0(t),z_0(t))$ of degree $141$.  Let $C$ be the (nonsingular projective model of the) hyperelliptic curve defined by $Y^2=h(t)$.  We have the four identities (where $x_0=x_0(t)$, $y_0=y_0(t)$, etc.),
\begin{align*}
\left(Y+(x_0^3+y_0^3-z_0^3)\right)\left(Y-(x_0^3+y_0^3-z_0^3)\right)=-4x_0^3y_0^3,\\
\left(Y+(x_0^3-y_0^3+z_0^3)\right)\left(Y-(x_0^3-y_0^3+z_0^3)\right)=-4x_0^3z_0^3,\\
\left(Y+(x_1^3+y_1^3-z_1^3)\right)\left(Y-(x_1^3+y_1^3-z_1^3)\right)=-4x_1^3y_1^3,\\
\left(Y+(-x_2^3+y_2^3+z_2^3)\right)\left(Y-(-x_2^3+y_2^3+z_2^3)\right)=-4y_2^3z_2^3.
\end{align*}
It follows that there are divisors $D_1$, $D_2$, $D_3$, and $D_4$ on $C$ such that 
\begin{align*}
(Y+x_0^3+y_0^3-z_0^3)=3D_1,\\
(Y+x_0^3-y_0^3+z_0^3)=3D_2,\\
(Y+x_1^3+y_1^3-z_1^3)=3D_3,\\
(Y-x_2^3+y_2^3+z_2^3)=3D_4.
\end{align*}
Using Magma, it is easy to verify that $D_1$, $D_2$, $D_3$, and $D_4$ give independent $3$-torsion elements of $\Jac(C)(\mathbb{Q})$ (to simplify calculations, this can be done modulo $p=7$, a prime of good reduction of $C$).  Thus, we arrive at the following result.
\begin{theorem}
Let $C$ be the  hyperelliptic curve defined by $Y^2=h(t)$.  Then $\rk_3 \Jac(C)(\mathbb{Q})_{\rm{tors}}\geq 4$.
\end{theorem}
Since $h$ has odd degree, $C$ has a rational Weierstrass point, and so Corollary~\ref{chyp} applies.
\begin{corollary}
\label{LCra2}
There are $\gg X^{\frac{1}{141}}/\log X$ imaginary quadratic fields $k$ with $|d_k|<X$ and $\rk_3 \Cl(k)\geq 4$ and $\gg X^{\frac{1}{141}}/\log X$ real quadratic fields $k$ with $d_k<X$ and $\rk_3 \Cl(k)\geq 3$. 
\end{corollary}

\subsection{$5$-ranks of quadratic fields and a construction of Mestre}
In \cite{Mes4}, Mestre showed that there exist infinitely many imaginary and real quadratic fields $k$ with $\rk_5 \Cl(k)\geq 3$.  Fitting Mestre's work into our framework, we show that a slight variation of Mestre's constructions yields a one-parameter family of hyperelliptic curves $C$ which have a rational Weierstrass point and $\rk_5 \Jac(C)(\mathbb{Q})_{\rm tors}\geq 3$.  Thus, by Corollary \ref{chyp}, we obtain infinitely many imaginary quadratic fields $k$ with $\rk_5 \Cl(k)\geq 3$ and infinitely many real quadratic fields $k$ with $\rk_5 \Cl(k)\geq 2$.  Unfortunately, we are not able to fully recover Mestre's result in the real quadratic case.  

Elliptic curves with a rational point of order $10$ form a one-parameter family, given explicitly by Kubert (after the change of variable $f=\frac{u+1}{2}$ in \cite{Kub}) as $y^2=g_u(x)$, where
\begin{align*}
g_u(x)&=(x^2-u(u^2+u-1))h_u(x),\\
h_u(x)&=8xu^2+(u^2+1)(u^4-2u^3-6u^2+2u+1).
\end{align*}
Alternatively, we can consider this as an elliptic curve over $\mathbb{Q}(u)$ defined by $y^2=g_u(x)$.  This has a point $P$ over $\mathbb{Q}(u)$ of order $10$ , given by
\begin{equation*}
P=\left(\frac{1+3u-u^2-u^3}{2u},\frac{(1-u)(1+u)^3(1+4u-u^2)}{2u}\right).
\end{equation*}
Let
\begin{align*}
u_1&=(t^2+t-1)/(t^2+t+1),\\
u_2&=-(t^2+3t+1)/(t^2+t+1),\\
u_3&=-(t^2-t-1)/(t^2+t+1).
\end{align*}
Then
\begin{equation*}
u_1(u_1^2+u_1-1)=u_2(u_2^2+u_2-1)=u_3(u_3^2+u_3-1).
\end{equation*}
Let $C$ be the nonsingular projective model of the curve (over $\mathbb{Q}(t)$) defined by
\begin{align}
y_1^2&=g_{u_1}(x),\label{C1}\\
y_2^2&=g_{u_2}(x),\label{C2}\\
y_3^2&=g_{u_3}(x).\label{C3}
\end{align}
\begin{theorem}
\label{Mh}
The curve $C$ is hyperelliptic of genus $5$ with a $\mathbb{Q}(t)$-rational Weierstrass point and $\rk_{10} \Jac(C)(\mathbb{Q}(t))_{\rm tors}\geq 3$.
\end{theorem}
\begin{proof}
Consider the three elliptic curves over $\mathbb{Q}(t)$ given (in affine equations) by
\begin{align*}
E_1: y^2&=g_{u_1}(x),\\
E_2: y^2&=g_{u_2}(x),\\
E_3: y^2&=g_{u_3}(x).
\end{align*}
Let $\phi_i$, $i=1,2,3$, be the natural projection map from $C$ onto $E_i$ induced by the map $(x,y_1,y_2,y_3)\mapsto (x,y_i)$.  Let $D_1$, $D_2$, and $D_3$ be divisors generating the $\mathbb{Q}(t)$-rational $5$-torsion on $E_1$, $E_2$, and $E_3$, respectively.  Let $\tilde{D}_i=\phi_i^*D_i$, $i=1,2,3$.  Then $\phi_{i*}(a_1\tilde{D_1}+a_2\tilde{D_2}+a_3\tilde{D_3})\sim (\deg \phi_i)a_iD_i=4a_iD_i$ on $E_i$, $i=1,2,3$, where $\sim$ denotes linear equivalence.  Since $(4,5)=1$, it follows immediately that $\tilde{D_1}$, $\tilde{D_2}$, and $\tilde{D_3}$ generate a $5$-torsion subgroup of rank $3$ in $\Jac(C)(\mathbb{Q}(t))$.

Consider the curve $C'$ over $\mathbb{Q}(t)$ that is the nonsingular projective model of the curve defined by $h_{u_1}(x)=v^2h_{u_2}(x)=w^2h_{u_3}(x)$.  Then we have a degree two map $\psi:C\to C'$ induced by the map 
\begin{equation*}
(x,y_1,y_2,y_3)\mapsto \left(x,v=\frac{y_1}{y_2},w=\frac{y_1}{y_3}\right).
\end{equation*}
We claim that $C'\cong \mathbb{P}^1_{\mathbb{Q}(t)}$. In fact, an explicit parametrization (in $z$) of the curve $h_{u_1}(x)=v^2h_{u_2}(x)=w^2h_{u_3}(x)$ is given by
\begin{align}
x&=\frac{(1 + u_1^2)(1 + 2 u_1 - 6 u_1^2 - 2 u_1^3 + 
                u_1^4) - (1 + u_2^2) (1 + 2 u_2 - 6 u_2^2 - 
                2 u_2^3 + u_2^4) v^2}{ 8 (u_2^2 v^2- u_1^2)
         }\label{zp1}\\
&=\frac{(1 + u_1^2)(1 + 2 u_1 - 6 u_1^2 - 2 u_1^3 + 
                u_1^4) - (1 + u_3^2) (1 + 2 u_3 - 6 u_3^2 - 
                2 u_3^3 + u_3^4) w^2}{8 (u_3^2 w^2-u_1^2) 
          },\label{zp2}\\
v&=\frac{-1 - t + 3 t^2 + 2 t^3 - 2 z - 4 t z + 4 t^2 z + 2 t^3 
z + z^2 - 2 t z^2 + t^3 z^2}{1 + 5 t + 7 t^2 + 2 t^3 - 2 z - 
2 t z + 6 t^2 z + 4 t^3 z - z^2 - 2 t z^2 + 2 t^2 z^2 + t^3 
z^2},\\
w&=-\frac{-1 - t + 3 t^2 + 2 t^3 - 2 z - 4 t z + 4 t^2 z + 
          2 t^3 z + z^2 - 2 t z^2 + 
          t^3 z^2}{(-1 - t + t^2)(-1 - 2 t - z^2 + 
              t z^2)}.\label{zp3}
\end{align}
An explicit hyperelliptic equation for $C$ can be obtained by substituting $x=x(z)$ from \eqref{zp1} into any one of \eqref{C1}, \eqref{C2}, or \eqref{C3}.  Note that the map $\psi$ is ramified above the four $\mathbb{Q}(t)$-rational points on $C'$ given by $(x,v,w)=(\infty, \pm \frac{u_1}{u_2},\pm\frac{u_1}{u_3})$.  It follows that $C$ has at least four rational Weierstrass points.  Their pairwise differences generate a $2$-torsion subgroup of rank $3$ in $\Jac(C)(\mathbb{Q}(t))$.  Finally, a calculation of the explicit hyperelliptic equation for $C$ (or a careful Riemann-Hurwitz calculation) shows that $C$ has genus $5$.
\end{proof}
For infinitely many specializations $t\in \mathbb{Z}$, we obtain a hyperelliptic curve over $\mathbb{Q}$ with the same properties as in Theorem \ref{Mh}. So by Corollary \ref{chyp}, we obtain the following.
\begin{corollary}
\label{LMes}
There are $\gg X^{\frac{1}{11}}/\log X$ imaginary quadratic fields $k$ with $|d_k|<X$ and $\rk_{10} \Cl(k)\geq 3$ and $\gg X^{\frac{1}{11}}/\log X$ real quadratic fields $k$ with $d_k<X$ and $\rk_{10} \Cl(k)\geq 2$.
\end{corollary}

\subsection{Higher degree fields}
A natural family of curves which, relatively speaking, have a lot of $m$-torsion in their Picard group are superelliptic curves $C$ of the form $y^m=a_0\prod_{i=1}^r(x-a_i)$, $a_0,\ldots, a_r\in \mathbb{Z}$.  We already took advantage of such curves in Section \ref{Rosen}, using the projection map $\phi:C\to \mathbb{A}^1$, $(x,y)\mapsto x$, in our proof of Theorem \ref{rn}.  These curves, implicitly, also form the basis for the work of Azuhata, Ichimura, and Nakano on ideal class groups in higher degree number fields.  In their work they use, essentially, the other simple projection map of these curves, $(x,y)\mapsto y$.  Applying a version of the technique of the present paper in this context, in \cite{Lev2} a quantitative version of a case of Azuhata and Ichimura's theorem was derived.
\begin{theorem}
\label{thLev1}
Let $m,n>1$ be positive integers.  There are $\gg X^{\frac{1}{m(n-1)}}/\log X$ number fields $k$ of degree $n$ with $|d_k|<X$ and
\begin{equation*}
\rk_m \Cl(k)\geq \left\lfloor \frac{n}{2}\right\rfloor.
\end{equation*}
\end{theorem}
Improving on this result, we derive a quantitative version of Nakano's theorem.
\begin{theorem}
\label{QN}
Let $m,n>1$ be positive integers.  Let $r_1$ and $r_2$ be nonnegative integers such that $r_1+2r_2=n$.  Let
\begin{equation*}
f(m,n,r_1,X)=
\begin{cases}
X^{\frac{1}{m(r_1-2)(2n-r_1+1)}}/\log X \quad &\mbox{ if $n$ is odd and $3<r_1<n$},\\
X^{\frac{1}{m(r_1-3)(2n-r_1+2)}}/\log X \quad &\mbox{ if $n$ is even and $2<r_1<n$},\\
X^{\frac{1}{2m(n-1)}}/\log X \quad &\mbox{ if $r_1=n$},\\
X^{\frac{1}{mn(n-1)}}/\log X \quad &\mbox{ if $m$ and $n$ are both even and $r_1=0$},\\
X^{\frac{1}{mn}}/\log X \quad &\mbox{ if $n$ is odd and $r_1=1, 3$ or}\\
&\mbox{ if $n$ is even and $r_1=2$ or}\\
&\mbox{ if $m$ is odd, $n$ is even, and $r_1=0$.}
\end{cases}
\end{equation*}
Then there are $\gg f(m,n,r_1,X)$ number fields $k$ of degree $n$ with $r_1$ real places and $r_2$ complex places, $|d_k|<X$, and
\begin{equation*}
\rk_m \Cl(k)\geq r_2+1.
\end{equation*}
\end{theorem}

\begin{proof}
Let $V\subset\mathbb{A}^{n+2}$ be the variety defined by $xt_0^m=-\prod_{j=1}^n (x-t_j^m)$.  Let $P=(x,t_0,\ldots, t_n)\in V(\Qbar)$, with $t_0,\ldots, t_n\in \mathbb{Z}$.  Let $k=\mathbb{Q}(P)$.
\begin{lemma}
\label{lrp}
Suppose that 
\begin{align}
&\left(t_0,\prod_{1\leq  i<j\leq n} t_i^m-t_j^m\right)=1,\label{rp1}\\
&\left(t_{i},t_0^m+(-1)^{n-1}\prod_{\substack{j=1\\j\neq i}}^nt_j^m\right)=1, \quad i=1,\ldots, n.\label{rp2}
\end{align}
Then for $j=1,\ldots, n$, $(x-t_j^m)=\mathfrak{a}_j^m$ for some ideal $\mathfrak{a}_j$ of $\co_k$.
\end{lemma}
\begin{proof}
Let $\alpha_j=x-t_j^m$, $j=1,\ldots, n$.  Let $i\in \{1,\ldots,n\}$.  Then from the definitions, $\alpha_{i}$ satisfies the equation
\begin{equation}
\label{peq}
\alpha_{i}\left(t_0^m+\prod_{\substack{j=1\\j\neq i}}^n (\alpha_{i}+t_{i}^m-t_j^m)\right)=-(t_0t_{i})^m.
\end{equation}
Note that by \eqref{rp1} and \eqref{rp2},
\begin{equation*}
\left((t_0t_i)^m,t_0^m+\prod_{\substack{j=1\\j\neq i}}^n (t_{i}^m-t_j^m)\right)=1.
\end{equation*}
Since $\alpha_i$ divides $(t_0t_i)^m$,
\begin{equation*}
\left(\alpha_{i},t_0^m+\prod_{\substack{j=1\\j\neq i}}^n (\alpha_{i}+t_{i}^m-t_j^m)\right)=\left(\alpha_{i},t_0^m+\prod_{\substack{j=1\\j\neq i}}^n (t_{i}^m-t_j^m)\right)=\co_k.
\end{equation*}
Using \eqref{peq}, we then obtain the lemma.
\end{proof}

We prove the theorem by looking at certain curves on $V$.  For instance, suppose that $n$ is odd and $3<r_1<n$ (for $r_1=3$, the proof also works, but in this case there is a better construction).  We consider an affine plane curve $C$ defined by (in $x$ and $y$) $xt_0^m=-\prod_{j=1}^n (x-t_j^m)$, where
\begin{align}
\label{teq1}
t_0&=(BMy+1)^{r_1-2},\\
\label{teq2}
t_i&=iBqy+a_i, &&i=1,\ldots, r_1-2,\\
\label{teq3}
t_i&=p^{i-r_1+2},&&i=r_1-1,\ldots, n,
\end{align}
and $B,M,p,q, a_1,\ldots, a_{r_1-2}\in \mathbb{Z}$.
\begin{lemma}
\label{nlem}
There exist choices of primes $p, q$, and integers $B,M,a_1,\ldots, a_{r_1-2}\in \mathbb{Z}$ such that for all $y\in \mathbb{Z}$, \eqref{rp1} and \eqref{rp2} hold for $t_i$ as in \eqref{teq1}--\eqref{teq3},
\begin{equation*}
p^m,p^{2m},\ldots,p^{m(n-r_1+2)},a_1^m,\ldots, a_{r_1-2}^m
\end{equation*}
are distinct and nonzero modulo $q$, $M$ is arbitrarily large compared to $p,q,a_1,\ldots, a_{r_1-2}$, and $BM\not\equiv 0 \pmod q$.
\end{lemma}
\begin{proof}
Let $b_i=a_iM-iq$, $i=1,\ldots, r_1-2$, and $b_i=Mp^{i-r_1+2}$, $i=r_1-1,\ldots, n$.  Let $B_0=\prod_{1\leq  i<j\leq n} b_i^m-b_j^m$.  For $i=1,\ldots, r_1-2$, let
\begin{equation*}
B_i=(iq-a_iM)^{m(r_1-2)}+(iq)^{m(r_1-2)}p^c\prod_{\substack{j=1\\ j\neq i}}^{r_1-2}(ia_j-ja_i)^m.
\end{equation*}
We choose $B$ such that $B=\prod_{i=0}^{r_1-2}B_i$.  We claim that for this choice of $B$, Eqs. \eqref{rp1} and \eqref{rp2} can be reduced to the equations
\begin{equation}
\label{rp3}
\left(a_i, (iq)\left(1+p^c\prod_{\substack{j=1\\ j\neq i}}^{r_1-2}(ia_j)^m\right)\right)=1, \quad i=1,\ldots, r_1-2,
\end{equation}
where $c={n-r_1+3 \choose 2}$.  We now show this.

Suppose that \eqref{rp3} holds.  The constant $B_0$ was chosen such that $B_0$ is in the ideal $\left(BMy+1,\prod_{1\leq  i<j\leq n} t_i^m-t_j^m\right)$.  Since $t_0=(BMy+1)^{r_1-2}\equiv 1 \pmod {B_0}$, \eqref{rp1} holds.  Similarly, for $i=1,\ldots, r_1-2$, $B_i$ is in the ideal $\left(t_{i},t_0^m+\prod_{j\neq 0,i}t_j^m\right)$ (note that since $n$ is odd, $(-1)^{n-1}=1$).  By \eqref{rp3}, $(a_i,B_i)=1$ for $i=1,\ldots, r_1-2$.  Since $t_i\equiv a_i \pmod{B_i}$, $i=1,\ldots, r_1-2$, we have that \eqref{rp2} holds for $i=1,\ldots, r_1-2$.  Since $r_1<n$, $B\equiv 0 \pmod p$ and  for any $i$, $t_0^m+\prod_{j\neq 0,i}t_j^m \equiv 1 \pmod p$.  Thus \eqref{rp2} holds for  $i=r_1-1,\ldots, n$ also.  

It is then fairly easy to see that there exists a choice of $B,M,p,q,a_1,\ldots, a_{r_1-2}\in \mathbb{Z}$ satisfying \eqref{rp3} as well as the other conditions in the statement of the lemma. 
\end{proof}
Choose $B,M,p,q,a_1,\ldots, a_{r_1-2}\in \mathbb{Z}$ as in Lemma \ref{nlem}, with $M$ sufficiently large compared to $p,q,a_1,\ldots, a_{r_1-2}$ (this will be quantified later).  Let $f_j$ be the rational function on $C$ given by $f_j=x-t_j$, $j=1,\ldots, n$.  Let $P=(x_0,y_0)\in C(\Qbar)$, with $y_0\in \mathbb{Z}$ and $k=\mathbb{Q}(P)$.  Then by Lemmas \ref{lrp} and \ref{nlem}, for $j=1,\ldots, n$, $f_j(P)\co_k=\mathfrak{a}_j^m$ for some ideal $\mathfrak{a}_j$ of $\co_k$.  Modulo the prime $q$, the equation defining $C$ becomes 
\begin{equation*}
x(BMy+1)^{m(r_1-2)}=-\prod_{j=1}^{r_1-2} (x-a_j^m)\prod_{j=r_1-1}^{n} (x-p^{m(j-r_1+2)}),
\end{equation*}
with $p^m,p^{2m},\ldots,p^{m(n-r_1+2)},a_1^m,\ldots, a_{r_1-2}^m$ distinct and nonzero modulo $q$ and $BM\not\equiv 0 \pmod q$. Thus, by looking modulo $q$, it follows easily that $C$ is in fact irreducible and
\begin{equation*}
\left[\Qbar(C)\left(\sqrt[m]{f_1},\ldots,\sqrt[m]{f_n}\right):\Qbar(C)\right]=m^n.
\end{equation*}
Let $\phi:C\to \mathbb{A}^1$ be the projection map $(x,y)\mapsto y$.  Note that $\deg \phi=n$.   For $i\in \mathbb{Z}$, let $P_i\in \phi^{-1}(i)$.  Then by Theorem \ref{thmain}, for all but $O(\sqrt{N})$ values $i=1\ldots, N$, $[\mathbb{Q}(P_i):\mathbb{Q}]=n$ and
\begin{equation*}
\rk_m \Cl(\mathbb{Q}(P_i))\geq n-\rk \co_{\mathbb{Q}(P_i)}^*.
\end{equation*}
Using Theorem \ref{Zan}, to finish the proof of Theorem \ref{QN} when $n$ is odd and $1<r_1<n$, it suffices to prove:
\begin{lemma}
For $i\gg 0$, if $[\mathbb{Q}(P_i):\mathbb{Q}]=n$ then $\mathbb{Q}(P_i)$ has exactly $r_1$ real places and 
\begin{equation*}
d_{\mathbb{Q}(P_i)}=O(i^{(r_1-2)m(2n-r_1+1)}).
\end{equation*}
\end{lemma}
\begin{proof}
We need to analyze the roots (in $x$) of the polynomial 
\begin{equation*}
g_y(x)=x(BMy+1)^{m(r_1-2)}+\prod_{i=1}^{n-r_1+2}(x-p^i)\prod_{j=1}^{r_1-2} \left(x-(jBqy+a_j)^m\right)
\end{equation*}
for $y\gg 0$.  We claim that for $y\gg 0$, $g_y(x)$ has $r_1-2$ real roots with absolute value $O(y^m)$, two real roots with absolute value $O(1)$, and $n-r_1$ nonreal complex roots with absolute value $O(1)$.  For $y\gg 0$, the polynomial $g_y(x)$ changes signs once on the interval $[y,(Bqy+a_1)^m]$ and twice on the intervals $[(2jBqy+a_{2j})^m,((2j+1)Bqy+a_{2j+1})^m]$, $j=1,\ldots, \frac{r_1-3}{2}$.  Thus, we obtain $r_1-2$ real roots of size $O(y^m)$.  More precisely, since 
\begin{equation*}
\lim_{y\to \infty}\frac{g_y(y^mx)}{y^{mn}}=x^{n-r_1+2}\prod_{j=1}^{r_1-2} \left(x-(jBq)^m\right),
\end{equation*}
for these $r_1-2$ real roots $\alpha_1<\alpha_2<\cdots <\alpha_{r_1-2}$, $\alpha_j\sim (jBqy)^m$ as $y\to \infty$.  Let $D=((r_1-2)!q^{r_1-2})^m$.  Then
\begin{equation*}
\lim_{y\to \infty}\frac{x^ng_y\left(\frac{1}{x}\right)}{(By)^{m(r_1-2)}}=x^{r_1-2}\left(-D\prod_{i=1}^{n-r_1+2}(1-p^ix)+ M^{m(r_1-2)}x^{n-r_1+1}\right).
\end{equation*}
It follows that the remaining $n-r_1+2$ roots of $g_y(x)$ converge to the reciprocals of the roots of $h(x)=\left(-D\prod_{i=1}^{n-r_1+2}(1-p^ix)+ M^{m(r_1-2)}x^{n-r_1+1}\right)$ as $y\to \infty$.  In particular, all of these roots are bounded in absolute value as $y\to \infty$.  A simple computation then shows that the discriminant of $g_y$ is $O(y^{(r_1-2)m(2n-r_1+1)})$, as desired.  We now determine the number of real roots of $h(x)$.  Looking at sign changes of $h(x)$, it is easily seen that $h(x)$ has at least two real roots.  Furthermore, note that
\begin{equation*}
\lim_{M\to\infty}\frac{h(M^{m(r_1-2)}x)}{M^{(r_1-2)m(n-r_1+2)}}=x^{n-r_1+1}(-Dp^cx+1)
\end{equation*}
and
\begin{equation*}
\lim_{M\to\infty}x^{n-r_1+2}h\left(\frac{1}{xM^{m(r_1-2)/(n-r_1+1)}}\right)=x(-Dx^{n-r_1+1}+1).
\end{equation*}
It follows from this that for $M\gg 0$, $h(x)$ has exactly two real roots.  More precisely, let $\epsilon>0$.  Then for all $M$ sufficiently large (not depending on $y$), as $y\to \infty$, among the remaining $n-r_1+2$ roots of $g_y(x)$, $\beta_1,\ldots,\beta_{n-r_1+2}$, exactly two of them are real and the roots satisfy (after reindexing)
\begin{equation*}
\left|\frac{\zeta^j}{\beta_j}\left(\frac{M^{m(r_1-2)}}{D}\right)^{\frac{1}{n-r_1+1}}-1\right|<\epsilon, \quad j=1,\ldots, n-r_1+1,
\end{equation*}
where $\zeta=e^{\frac{2\pi i}{n-r_1+1}}$, and $\left|\frac{Dp^c}{\beta_{n-r_1+2}M^{m(r_1-2)}}-1\right|<\epsilon$.  Therefore, having fixed a sufficiently large integer $M$, we see that $g_y(x)$ has exactly $r_1$ real roots for $y\gg 0$.
\end{proof}

The proofs of the other cases in the theorem are similar (or even easier), so we give only a few details.  When $n$ is even and $2<r_1<n$, similar to the above, we let $t_0=(BMy+1)^{r_1-3}$, $t_i=iBqy+a_i$, $i=1,\ldots, r_1-3$, and $t_i=p^{i-r_1+3}$, $i=r_1-2,\ldots, n$.  The proof is then essentially identical to the proof above for $n$ odd and $3<r_1<n$.

To prove the cases $n$ odd, $r_1=1$, and $n$ even, $r_1=2$, we consider curves $C$ of the form $x(ay+b)^m=-\prod_{i=1}^n(x-c_i^m)$.  We choose $a,b,c_1,\ldots, c_n$ to be positive integers such that, for any $y\in \mathbb{Z}$, \eqref{rp1} and \eqref{rp2} hold with $t_0=ay+b$ and $t_i=c_i$, $i=1,\ldots, n$.  For $y\gg 0$, the polynomial $x(ay+b)^m+\prod_{i=1}^n(x-c_i^m)$ has exactly one real root if $n$ is odd, exactly two real roots if $n$ is even, and discriminant that is $O(y^{mn})$.  So the desired result follows from Lemma \ref{lrp}, Theorem \ref{Zan}, and Theorem \ref{thmain}, with $V=C$, ${\bf I}=\mathbb{Z}$, $\phi=y$, and $f_i=x-c_i^m$, $i=1\ldots, n$.  The case $n$ odd, $r_1=3$, is similar, except that we consider curves of the form $x(ay+b)^m=\prod_{i=1}^n(x-c_i^m)$.  The case $m$ odd, $n$ even, and $r_1=0$, is similar except that we consider curves of the form $x^2(ay+b)^m=-\prod_{i=1}^n(x-c_i^m)$ (in the case that $m$ and $n$ are both even, the proof breaks down because $-\prod_{i=1}^n(x-c_i^m)=(x(ay+b)^\frac{m}{2})^2$ and so the needed version of \eqref{cond1} does not hold).

In the case $r_1=n$, we consider curves $C$ of the form 
\begin{equation*}
g_y(x)=x(a_0y+b_0)^m+(x-(My+b_1)^m)\prod_{i=2}^n (x-b_i^m)=0, 
\end{equation*}
with $M$ sufficiently large compared to $a_0,b_0,\ldots, b_n$.  With an appropriate choice of integers $a_0,b_0,\ldots, b_n,M$, we can apply Theorem \ref{thmain} to $C$ with $f_1=x-(My+b_1)^m$, $f_i=(x-b_i^m)$, $i=2,\ldots, n$, and $\phi:C\to\mathbb{A}^1$, $(x,y)\mapsto y$.  For $M$ sufficiently large, $g_y(x)$ has $n$ real roots for $y$ sufficiently large and discriminant (in $x$) that is $O(y^{2m(n-1)})$.

Finally, when $m$ and $n$ are both even and $r_1=0$, we consider curves $C$ of the form
\begin{equation*}
g_y(x)=x(a_0y+b_0)^{m(n-1)}+\prod_{i=1}^n (x-(My+b_i)^m),
\end{equation*}
with $M$ sufficiently large compared to $a_0,b_0,\ldots,b_n$.  Again, with an appropriate choice of integers $a_0,b_0,\ldots, b_n,M$, we can apply Theorem \ref{thmain} to $C$ with $f_i=(x-(My+b_i)^m)$, $i=1,\ldots, n$, and $\phi:C\to\mathbb{A}^1$, $(x,y)\mapsto y$.  For $M$ sufficiently large, $g_y(x)$ has no real roots for $y$ sufficiently large and discriminant (in $x$) that is $O(y^{m(n-1)n})$.
\end{proof}

For superelliptic curves $C$ of the form $y^m=a_0\prod_{i=1}^r(x-a_i)$, $a_0,\ldots, a_r\in \mathbb{Z}$, we have discussed applying Theorem \ref{thmain} to $C$ using the maps $\phi_1,\phi_2:C\to \mathbb{A}^1$ given by $\phi_1=x$ and $\phi_2=y$.  Using other maps, it was shown in \cite{Lev2} that in some situations it is possible to improve on Nakano's inequality \eqref{eqN1}.

\begin{theorem}
\label{thLev2}
Let $m,n>1$ be positive integers with $n>(m-1)^2$.  There are $\gg X^{\frac{1}{(m+1)n-1}}/\log X$ number fields $k$ of degree $n$ with $|d_k|<X$ and
\begin{equation*}
\rk_m \Cl(k)\geq \left\lceil \left\lfloor\frac{n+1}{2}\right\rfloor+\frac{n}{m-1}-m\right\rceil.
\end{equation*}
\end{theorem}
Since a complete proof of this theorem is given in \cite{Lev2}, we only give a sketch of the proof.
\begin{proof}
Let $m,n>1$ be integers with $n>(m-1)^2$.  Let $r$ be the largest integer such that $r-\left\lfloor\frac{r}{m}\right\rfloor\leq n$ and $(r,m)=1$.  It is easily checked that $r\geq n+\frac{n}{m-1}-m+1$.  Let $C$ be the curve defined by $y^m=g(x)=-(x-a_r^m)\prod_{j=1}^{r-1}(x+a_i^m)$, where $a_1,\ldots, a_r$ are certain carefully chosen integers.  Let $f_j=x+a_{j}^m$, $j=1,\ldots, r-1$.  Then \eqref{cond1} holds (with $V=C$ and $r\mapsto r-1$).  Let $h(x)$ be the Taylor series for $\sqrt[m]{g(x)}$ at $x=0$ truncated to degree $\left\lfloor\frac{r}{m}\right\rfloor-1$ with $h(0)=\prod_{j=1}^ra_j$.  Let $b$ be the lowest common denominator of the coefficients of $h$.  Let $\phi$ be the rational function $\phi=\frac{b(y-h)}{x^{r-n}}$ on $C$.  For $i\in \mathbb{Z}$, let $P_i\in \phi^{-1}(i)$.  Having chosen $a_1,\ldots, a_r$ properly, it can be shown that there exist integers $i_0$ and $N$ such that if $i\equiv i_0 \pmod N$, then for all $j$, $f_j(P_i)=\mathfrak{a}_{i,j}$ for some fractional ideal $\mathfrak{a}_{i,j}$ of $\co_{\mathbb{Q}(P_i)}$.  So \eqref{cond2} holds with ${\bf I}=\{i\in \mathbb{Z}\mid i\equiv i_0 \pmod N\}$.  Furthermore, it can be computed that the map induced by $\phi$ has degree $n$, $d_{\mathbb{Q}(P_i)}=O(i^{(m+1)n-1})$, and $\mathbb{Q}(P_i)$ has at most two real places for $i\gg 0$.  An appropriate application of Theorems \ref{Zan} and \ref{thmain} finishes the proof.
\end{proof}
\bibliography{class2}
\end{document}